\documentclass[reqno,12pt]{amsart}
\usepackage{amssymb, amsmath,latexsym,amsfonts,amsbsy, amsthm}
\usepackage{color}
\usepackage{txfonts}
\usepackage{framed}
\usepackage{hyperref}
\newtheorem{theorem}{Theorem}[section]

\newtheorem{lemma}[theorem]{Lemma}
\newtheorem{proposition}[theorem]{Proposition}

\newtheorem{remark}[theorem]{Remark}

\def\bR{\mathbb R}

\def\bZ{\mathbb Z}

\def\bT{\mathbb T}

\def\bT{\mathbb{T}}

\def\md{\mathrm{d}}

\def\mK{\mathcal{K}}
\def\mI{\mathcal{I}}
\def\fK{\mathfrak{K}}
\def\fI{\mathfrak{I}}

\def\mM{\mathcal{M}}

\def\la{\lambda}

\def\t{\tilde}
\def\q{\quad}

\def\th{\theta}

\def\dl{\delta}
\def\Dl{\Delta}
\def\lt{\left}
\def\les{\lesssim}
\def\rt{\right}

\def\i{\infty}
\def\e{\epsilon}

\def \ls{\lesssim}
\def\p{\partial}
\def\f{\frac}
\def\na{\nabla}
\def\al{\alpha}

\def\O{\Omega}

\def\s{\sqrt}
\def\bl{\boldsymbol}
\def\cd{\cdot}
\def\nn{\nonumber}

\def\be{\begin{equation}}
\def\ee{\end{equation}}
\def\bes{\begin{equation*}}
\def\ees{\end{equation*}}
\def\bali{\begin{aligned}}
\def\eali{\end{aligned}}
\def\beas{\begin{eqnarray*}}
\def\eeas{\end{eqnarray*}}
\def\bs{\begin{split}}
\def\es{\end{split}}

\allowdisplaybreaks[4]
\numberwithin{equation}{section}

\begin{document}
\title[Existence and non-uniqueness of stationary ASNS]
{Existence and non-uniqueness  of classical solutions to the axially symmetric stationary Navier-Stokes equations in an exterior cylinder}

\author{Zijin Li}
\address[Z. Li]{School of Mathematics and Statistics, Nanjing University of Information Science and Technology, Nanjing 210044, China, and Academy of Mathematics \& Systems Science, Chinese Academy of Sciences, Beijing 100190, China}
\email{zijinli@nuist.edu.cn}

\author{Xinghong Pan}
\address[X. Pan]{School of Mathematics and Key Laboratory of MIIT, Nanjing University of Aeronautics and Astronautics, Nanjing 211106, China}
\email{xinghong\_87@nuaa.edu.cn}


\subjclass[2020]{35Q30, 76D05.}

\keywords{stationary Navier-Stokes equations, exterior domain, existence, non-uniqueness.}

\thanks{Z. Li is supported by the China Postdoctoral Science Foundation (No. 2024M763474) and the National Natural Science Foundation of China (No. 12001285). X. Pan is supported by the National Natural Science Foundation of China (No. 12031006, No. 12471222).}

\maketitle

\begin{abstract}
  In this paper, we show existence and non-uniqueness on the axially symmetric stationary Navier-Stokes equations in an exterior periodic cylinder. On the boundary of the cylinder, the horizontally swirl velocity is subject to the perturbation of a rotation, the horizontally radial velocity is subject to the perturbation of an interior sink, while the vertical velocity is the perturbation of zero. At infinity, the flow stays at rest. We construct a solution to such problem, whose principal part admits a critical decay for the horizontal components and a supercritical decay for the vertical component of the velocity.

  This existence result is related to the 2D Stokes paradox and an open problem raised by V. I. Yudovich in [{\it Eleven great problems of mathematical hydrodynamics}, Mosc. Math. J. 3 (2003), no. 2, 711--737], where Problem 2 states that: {\em Show (spatially) global existence theorems for stationary and periodic flows.} Moreover, if the horizontally radial-sink velocity is relatively large ($\nu<-2$ in our setting), then the solution to this problem is non-unique.

\end{abstract}

\tableofcontents

\indent

\section{Introduction}

\subsection{Motivation}

\indent We consider the 3D stationary Navier-Stokes flow in an exterior  periodic cylinder $\O:=(\bR^2-B)\times\mathbb{T}$, where $B=\{x_h=(x_1,x_2)\in\mathbb{R}^2\,:\,x_1^2+x_2^2<1\}$ is the unit disk and $\bT=\{x_3\in\bR:x_3\in[0,2\pi[\}$ is the torus with $2\pi$ periodicity.
\begin{equation}\label{NS}
\left\{\begin{array}{ll}
-\Delta \bl{u}+\bl{u} \cdot \nabla \bl{u}+\nabla p=\bl{f}, & x\in\O\,,\\[1mm]
\nabla \cdot \bl{u}=0, & x\in\O\,,\\[1mm]
\bl{u}=\bl{u}_\ast, & x\in\p \O\,,\\[1mm]
\bl{u}=0,&  |x_h|\rightarrow+\i.
\end{array}
\rt.
\end{equation}
Here $\bl{u}(x)=(u_1,u_2,u_3):\O\rightarrow \bR^3$ is the unknown velocity of the fluid, while $p$ is the scalar pressure. $\bl{f}$ is the external force. $\bl{u}_\ast$ is the boundary value.

We show the existence result of system \eqref{NS} which was initially motivated by considering the existence problem of the 2D exterior-domain problem for system \eqref{NS} by ignoring the periodic variable $x_3$ and the vertical velocity component $u_3$ in an exterior domain, which is stated more clearly as follows. To find a solution to the following problem
\begin{equation}\label{2dns}
\left \{
\begin {array}{ll}
\bl{u}\cdot\na \bl{u}+\na p-\Dl \bl{u}=0, & \text{in }\bR^2-D\,,\\ [5pt]
\na\cdot\bl{u}=0, & \text{in }\bR^2-D\,,\\
\bl{u}\big|_{\p D}=\bl{a}_\ast,\\
\bl{u}\big|_{|x|\rightarrow+\i}=\eta\bl{e}_1,\\
\end{array}
\right.
\end{equation}
where $D$ is a smooth bounded domain and $\bl{a}_\ast$ is a smooth function defined on $\p D$. The existence problem of \eqref{2dns} catches the mathematicians' attention since the Stokes paradox, which show that the linear Stokes equation of \eqref{2dns} with $\bl{a}_\ast=\bl{0}$ and $\eta\neq 0$ has no solution. For the nonlinear Navier-Stokes system, such an existence problem of \eqref{2dns} with general $\bl{a}^\ast$ and $\eta$ was still an open problem and was listed  as one of the ``Eleven Great Problems in Mathematical Hydrodynamics" (Problem 2) by Yudovich in \cite{Yudovich:2003MMJ}.  Leray in \cite{Leray:1933JMPA} firstly developed the {\em invading domains method} to study this 2D existence problem. By using Leray's method, a $D$-solution (the solution with finite Dirichlet integration) satisfying \eqref{2dns}$_{1,2,3}$ and no flux condition $\int_{\p D} \bl{a}^\ast \cdot \bl{n} dS=0$ can be obtained in \cite{KorobkovPR:2020JDE}. However, whether this $D$-solution satisfies \eqref{2dns}$_4$ is unknown. Also, there is no result of $D$-solution if the flux of the flow is non-zero, even only satisfying \eqref{2dns}$_{1,2,3}$. The main difficulties of the 2D existence problem are the following: The lack of Sobolev embedding in two dimensions and the logarithmic growth of the Green tensor for the 2D Stokes system. Despite the above difficulties, under the assumptions that $|\bl{a}_\ast-\eta\bl{e}_1|$ is small, Finn and Smith in \cite{FinnS:1967ARMA} gave an existence result by using iteration techniques. Whether the Finn-Smith solution is a $D$-solution is still unknown. Recently Korobkov-Ren \cite{KorobkovR:2021ARMA, KorobkovR:2022JMPA} shows existence and uniqueness of $D$-solutions in the case that $\bl{a}_\ast=0$ and $\eta$ is small. See recent advances on this topic in \cite{KorobkovPR:2019ARMA,KorobkovPR:2021ADV} and references therein.

Based on the difficulties mentioned above for the existence of the 2D exterior problem, direct consideration of system \eqref{NS} with general data $\bl{f}$ and $\bl{u}_\ast$ is still far from being expected. We consider the data $\bl{f}$ and $\bl{u}_\ast$ are axially symmetric and show existence of the solution to system \eqref{NS} with the axial symmetric property which is intrinsic for the domain.

 Our problem in this paper will be mainly studied in the cylindrical coordinates $(r,\th,z)$, where
 \[
 r=\sqrt{x_1^2+x_2^2},\q \th=\arctan\f{x_2}{x_1},\q z=x_3\,.
 \]
We say a function $\bl{v}$ is axially symmetric if and only if
\[
\bl{v}=v_r(r,z)\bl{e_r}+v_\th(r,z)\bl{e_\th}+v_z(r,z)\bl{e_z}.
\]
Here the basis
\[
\bl{e_r}=\left(\frac{x_1}{r}, \frac{x_2}{r}, 0\right), \quad \bl{e_\theta}=\left(-\frac{x_2}{r}, \frac{x_1}{r}, 0\right), \quad \bl{e_z}=(0,0,1) .
\]
The boundary value of $\bl{u}_\ast$ is given as follows
\be\label{bvalue}
\bl{u}_\ast= \nu \bl{e}_r+ \mu\bl{e}_\th+ \bl{g}, \q \text{on } \p B\times \bT,
\ee
where $\nu<0$ and $\mu\in \bR$ are two constants to represent the interior sinks and rotations at the boundary. While $\bl{g}$  is axially symmetric perturbation function with suitable smoothness and small amplitude. Then system \ref{NS} with the boundary value \eqref{bvalue} is the following
\begin{equation}\label{ns0}
\left\{\begin{array}{ll}
-\Delta \bl{u}+\bl{u} \cdot \nabla \bl{u}+\nabla p=\bl{f}, & x\in\O\,,\\[1mm]
\nabla \cdot \bl{u}=0, & x\in\O\,,\\[1mm]
\bl{u}=(\nu+g_r(x_3))\bl{e}_r+(\mu+g_\th(x_3))\bl{e}_\th+g_z(x_3)\bl{e_z}, & x\in\p \O\,,\\[1mm]
\bl{u}=0,&  |x_h|\rightarrow+\i.
\end{array}
\rt.
\end{equation}
where $\bl{g}=g_r(z)\bl{e_r}+g_\th(z)\bl{e_\th}+g_z(z)\bl{e_z}$ is the axially symmetric representation of $\bl{g}$ on $\p B\times\bT$.
 Furthermore, by denoting the velocity and the external force in cylindrical coordinates by
\[
\bl{u}=u_r(r,z)\bl{e_r}+u_\th(r,z)\bl{e_\th}+u_z(r,z)\bl{e_z},\q \bl{f}=f_r(r,z)\bl{e_r}+f_\th(r,z)\bl{e_\th}+f_z(r,z)\bl{e_z}, \q\text{in } \O,
\]
then we can reformulate system  \eqref{ns0} as  follows
\be\label{ASNS}
\left\{\begin{split}
&(u_r\p_r+u_z\p_z) u_r-\frac{u_\theta^2}{r}+\partial_r p=\left(\Delta-\frac{1}{r^2}\right) u_r+f_r, \\
&(u_r\p_r+u_z\p_z)u_\theta+\frac{u_\theta u_r}{r}=\left(\Delta-\frac{1}{r^2}\right) u_\theta+f_\th, \\
&(u_r\p_r+u_z\p_z)u_z+\partial_z p=\Delta u_z+f_z, \\
&\partial_r u_r+\frac{u_r}{r}+\partial_z u_z=0\,\\
&(u_r,u_\th,u_z)(1,z)=\left(\nu+g_r(z),\mu+g_\th(z),g_z(z)\right),\\
&(u_r,u_\th,u_z)(+\i,z)=0.
\end{split}\right.\q\text{for}\q (r,z)\in]1,\i[\,\,\times\,\,\mathbb{T}\,,
\ee
Without loss of generality, we set $g_{r,0}:=\f{1}{2\pi}\int^{2\pi}_0 g(x_3)dx_3=0$. Or we can replace $\nu+g_r(x_3)$ by $\t{\nu}+\t{g}_r(x_3)$ with
\bes
\t{\nu}:=\nu+\f{1}{2\pi}\int^{2\pi}_0 g(x_3)dx_3,\q \t{g}_r(x_3):=g_r(x_3)-\f{1}{2\pi} \int^{2\pi}_0g(x_3)dx_3.
\ees

\subsection{Functional spaces and the main result}
Before stating the main theorem in this paper, we need to define the functional spaces where we work. First for a periodic function $h(z)\,:\,\mathbb{T}\to\mathbb{R}$, define its Fourier series by
\[
h(z)=\sum_{k\in\mathbb{Z}}h_k e^{ikz},
\]
where the sequence $\{h_k\}_{k\in\mathbb{Z}}$ is the set of its Fourier coefficients with
\[
h_k=\f{1}{2\pi}\int^{2\pi}_0 h(z) e^{-ik\th}\md z\,.
\]
 Also for a function $s(r):\, [1,+\i)\to\bR$, define its weighted $L^\i$ norm by
\[
\|s(\cd)\|_{L^\i_\zeta}:=\mathrm{esssup}_{r\in\,]1,\i[}\,r^\zeta |s(r)|\,.
\]

We first define the functional spaces for the external force $\bl{f}$ and the boundary value $\bl{g}$. Given constants $\la_\th>3$, $\la_z>2$, $\la>3/2$, for
\[
\bl{f}=f_r(r,z)\bl{e_r}+f_\th(r,z)\bl{e_\th}+f_z(r,z)\bl{e_z}\,:\,[1,\i[\,\times\mathbb{T}\to\mathbb{R}^3\,,
\]
we denote
\[
\begin{split}
\mathcal{E}_{\la_\th,\la_z,\la}&:=\left\{\bl{f}\,\,:\,\,\|\bl{f}\|_{\mathcal{E}_{\la_\th,\la_z,\la}}:=\|f_{\th,0}(r)\|_{L^\i_{\la_\th}}+\|f_{z,0}(r)\|_{L^\i_{\la_z}}+\sum_{k\neq0,j\in\{r,\th,z\}}\|f_{j,k}(r)\|_{L^\i_\la}<\i\right\}.\\
\end{split}
\]
Here we mention that there is no restriction on $f_{r,0}$ since we can absorb it in the zero mode of pressure. See Section \ref{sec4} below. Meanwhile, for the boundary value
\[
\bl{g}=g_r(z)\bl{e_r}+g_\th(z)\bl{e_\th}+g_z(z)\bl{e_z}\,:\,\mathbb{T}\to\mathbb{R}^3\,,
\]
we define
\[
\mathcal{V}:=\left\{\bl{g}\,:\,\mathbb{T}\to\mathbb{R}^3\big|\,g_{r,0}=0\,,\,\,\|\bl{g}\|_{\mathcal{V}}:=\sum_{k\in\mathbb{Z},j\in\{r,\th,z\}}(1+k^2)|g_{j,k}|<\i\right\}.
\]

Now we define the functional space for the velocity. For some $\tau>0$, we denote
\[
\mathcal{B}^0_{\tau}=\left\{\bl{v}_0=\lt(\mathrm{v}(r)+\f{\sigma}{r}{\bf{1}}_{-2\leq\nu<0}\rt)\bl{e_\th}+v_{z,0}(r)\bl{e_z}\,:\,\sum_{\ell\in\{0,1,2\}}\left(\|\mathrm{v}^{(2-\ell)}(r)\|_{L^\i_{3+\tau-\ell}}+\|v^{(2-\ell)}_{z,0}(r)\|_{L^\i_{2+\tau-\ell}}\right)<\i\right\}\,,
\]
which is designed for the zero mode of the reduced solution \footnote{That is, by subtracting the background solution $\f{\nu}{r}\bl{e_r}+\f{\mu}{r}\bl{e_\th}$ from the solution $\bl{u}$.}, where ${\bf{1}}_{-2\leq\nu<0}$ is the characteristic function on ${\nu\in[-2,0[}\,$, and $\sigma\in\mathbb{R}$. Meanwhile
\[
\mathcal{B}^{\neq}_{\tau}:=\left\{\bl{v}_\neq=\sum_{k\neq0,j\in\{r,\th,z\}}v_{j,k}(r)e^{ikz}\bl{e_j}\,:\,\mathrm{div}\,\bl{v}_\neq=0,\,\sum_{k\neq0,j\in\{r,\th,z\},\ell\in\{0,1,2\}}|k|^{2-\ell}\|v^{(\ell)}_{j,k}(r)\|_{L^\i_{3/2+\tau}}<\i\right\}\,,
\]
which is defined for the nonzero mode of the reduced solution. Finally, we define
\[
\mathcal{B}_{\tau}:=\mathcal{B}^0_{\tau}\bigoplus\mathcal{B}^{\neq}_{\tau}=\left\{\bl{v}=\bl{v}_0+\bl{v}_\neq: \bl{v}_0\in \mathcal{B}^0_{\tau},\, \bl{v}_\neq\in\mathcal{B}^{\neq}_{\tau} \right\}\,,
\]
and related norm is given by
\[
\begin{split}
\|\bl{v}\|_{{\mathcal{B}_{\tau}}}:=&|\sigma|{\bf1}_{-2\leq\nu<0}+\sum_{\ell\in\{0,1,2\}}\|\mathrm{v}^{(2-\ell)}(r)\|_{L^\i_{3+\tau-\ell}}+\sum_{k\neq0,\ell\in\{0,1,2\}}|k|^{2-l}\|v^{(\ell)}_{\th,k}(r)\|_{L^\i_{3/2+\tau}}\\
                        &+\sum_{\ell\in\{0,1,2\}}\|v^{(2-\ell)}_{z,0}(r)\|_{L^\i_{2+\tau-\ell}}+\sum_{k\neq0\atop j\in\{r,z\},\ell\in\{0,1,2\}}|k|^{2-\ell}\|v^{(l)}_{j,k}(r)\|_{L^\i_{3/2+\tau}}\,.
\end{split}
\]
Below is our main theorem:
\begin{theorem}\label{Main}
For fixed $\nu<0,\,\mu\in \bR$ and $\la_\th>3$, $\la_z>2$, $\la>3/2$, there exists $\e,\,\tau>0$ depending on the aforementioned constants, such that if
\[
\|\bl{f}\|_{\mathcal{E}_{\la_\th\,,\,\la_z\,,\,\la}}+\|\bl{g}\|_{\mathcal{V}}<\e\,,
\]
then the system \eqref{ns0} has a solution $\bl{u}$ such that
\[
\bl{u}=\f{\nu}{ r}\bl{e}_r+\f{{\mu}}{r}\bl{e}_\th+\bl{v}\,,
\]
where $\bl{v}\in{\mathcal{B}_{\tau}}\,$ satisfying
\[
\|\bl{v}\|_{{\mathcal{B}_{\tau}}}\leq C\left(\|\bl{f}\|_{\mathcal{E}_{\la_\th\,,\,\la_z\,,\,\la}}+\|\bl{g}\|_{\mathcal{V}}\right)< C\e,
\]
for some constants $C>0$, depending on $\mu,\,\nu,\,\la_\th,\,\la_z,\,\la$ mentioned above.

Moreover, if $\nu<-2$, we can construct infinite many classical solutions to system \eqref{ns0}.
\end{theorem}

$\hfill\square$

\begin{remark}
By the proof the main theorem, the index $\tau$ in the space of the solution could be chosen as
\[
\tau=\min\left\{{\la_\th}{\bf{1}}_{-2\leq\nu<0}+\min\left\{\la_\th\,,\,2-\f{\nu}{2}\right\}{\bf{1}}_{\nu<-2}-3\,,\,\min\left\{\la_z\,,\,2-\f{\nu}{2}\right\}-2\,,\,\la-\f{3}{2}\right\}\,.
\]
\end{remark}

$\hfill\square$

\begin{remark}
 Roughly speaking, our theorem \ref{Main} states that if the external force $\bl{f}\in C(\O)$ has small amplitude that satisfies suitable decay at spacial infinity, and the boundary value perturbation $\bl{g}\in C^2(\bT)$ is small enough, we can construct a $C^2(\O)$ solution to system \eqref{ns0}.
\end{remark}
$\hfill\square$
\begin{remark}
Actually, from the proof of Theorem \ref{Main} in the case of $\nu<-2$, we can choose $\t{\mu}$, which is different but sufficiently close to $\mu$, such that
\[
\t{\bl{u}}=\f{\nu}{ r}\bl{e}_r+\f{\t{\mu}}{r}\bl{e}_\th+\t{\bl{v}},\q \t{\bl{v}}\in \mathcal{B}_{\tau}\,,
\]
%
 is still a solution of system \eqref{ns0}. Since $\bl{v},\,\t{\bl{v}}\in\mathcal{B}_\tau$, we have
  \[
  v_\th\,,\,\t{v}_\th\in o(r^{-1})\,,\q\text{as}\q r\to\i\,.
  \]
Thus we find $\bl{u}$ and $\t{\bl{u}}$ are indeed two distinct solutions of \eqref{ns0} since their $r^{-1}$ coefficients are different. However, whether there is non-uniqueness for the problem \eqref{NS} with $\nu\geq-2$ is still unknown to the authors.

As far as the authors know, the only known nonuniqueness result to the stationary axially symmetric Navier-Stokes equations comes from Galid \cite [Page 596, Theorem IX. 2.2]{Galdi:2011}, where the author show that when considering the domain $\O=(B_{R_2}-B_{R_1})\times \bT$: There exist smooth $\bl{f}$, $\bl{u}_\ast$ and $\al$ such that the following system have at least two axially symmetric solutions.
\[
\left\{\begin{array}{ll}
-\al\Delta \bl{u}+\bl{u} \cdot \nabla \bl{u}+\nabla p=\bl{f}, & x\in\O\,,\\[1mm]
\nabla \cdot \bl{u}=0, & x\in\O\,,\\[1mm]
\bl{u}=\bl{u}_\ast, & x\in\p \O.
\end{array}
\rt.
\]
The idea of proof is inspired by Yudovich \cite{Yudovich:1967MATHSB}.
\end{remark}
$\hfill\square$
\begin{remark}
Mitsuo Higaki \cite{Higaki:ARXIV24} consider the  weak solution of the case $\nu<-2$, $\mu=0$, $\bl{g}=0$, and $\bl{f}=O(\f{1}{r^{3^+}})$. Our result improves the result in\cite{Higaki:ARXIV24} by:
(i) classical solutions; (ii) $\nu<0$, $\mu$ can be arbitrarily large; (iii) perturbation of boundary value; (iv) milder decay for $\bl{f}$; and (v) non-uniqueness for the case $\nu<-2$ as stated above.

\end{remark}

$\hfill\square$

\subsection{Strategy of proof to the main result}

\q\ There is an explicit solution $\f{\nu}{r} \bl{e}_r+\f{\mu}{r}\bl{e}_\th$ to system \eqref{ns0} if the external force $\bl{f}$ and the perturbation $\bl{g}$ are zero.  Such a solution is a rotation flow compounded with a sink flow, which is invariant under the natural scaling of the Navier-Stokes equations: $\bl{u}^{\al}(x):=\al\bl{u}(\al x)$, $\forall\al>0$. A scaling-invariant solution is called scale-critical, and it represents the balance between the nonlinear and linear parts of the equations. Given this nature, perturbation around a scaling-invariant solution may be complicated based on the scale of the perturbation. However, it is expected that the problem is well-posed if the perturbation is subcritical, which means the solution decays faster than order one of $r$ at spacial infinity. Based on such expectation, we construct a solution of the following type:
 \bes
 \bl{u}=\f{\nu}{r} \bl{e}_r+\f{\mu}{r}\bl{e}_\th+\bl{v},\q \text{with } |\bl{v}|=o\lt(r^{-1}\rt), \text{ as } r\rightarrow\i.
 \ees
By subtracting the scale-critical term
$
\f{\nu}{r}\bl{e_r}+\f{\mu}{r}\bl{e_\th}
$
from \eqref{ASNS}, the reduced solution $\bl{v}:=\bl{u}-\left(\f{\nu}{r}\bl{e_r}+\f{\mu}{r}\bl{e_\th}\right)$ satisfies
\be\label{ReduNS}
\left\{\begin{split}
&-\left(\p_r^2+\f{1-\nu}{r}\p_r-\f{1-\nu}{r^2}+\p_z^2\right)v_r+\p_r\pi=-\left({v}_r\p_r+{{v}_z}\p_z\right){v}_r+\f{{v}_\th^2}{r}+{\f{2\mu}{r^2}{v}_\th}+f_r\,,\\
& -\left(\p_r^2+\f{1-\nu}{r}\p_r-\f{1+\nu}{r^2}+\p_z^2\right)v_\th=-\left({v}_r\p_r+{{v}_z}\p_z\right){v}_\th-\f{{v}_r{v}_\th}{r}+f_\th\,,\\
&-\left(\p_r^2+\f{1-\nu}{r}\p_r+\p_z^2\right)v_z+\p_z\pi=-\left({v}_r\p_r+{{v}_z}\p_z\right){v}_z+f_z\,,\\
&\p_z(r v_z)+\p_r(rv_r)=0\,,\\
&v_r\big|_{r=1}={g}_r,\q v_\th\big|_{r=1}= g_\th,\q v_z\big|_{r=1}= g_z, \q \bl{v}\big|_{r\rightarrow+\i}=0\,,
\end{split}\right.
\ee
for $(r,z)\in]1,\i[\times\mathbb{T}$ .

The existence result of system \eqref{ReduNS} is based on the representation formula for the linearized system in Fourier mode. See \eqref{elinear}. In each Fourier mode of the linear system, solving for the horizontally swirl velocity component $v_\th$ is relatively straightforward, as its governing equation \eqref{elinear}$_2$ does not involve a pressure term. However, the situation becomes more subtle for the Fourier modes corresponding to the horizontally radial velocity component $v_r$  and the vertical component $v_z$, due to the presence of pressure. Nevertheless, for the Fourier zero mode of $v_r$ and $v_z$, it can be solved directly by using the equation of $v_{z,0}$ and the incompressibility. While for the Fourier non-zero mode of $v_r$ and $v_z$, we eliminate the pressure term by reformulating the problem in terms of the vorticity equation, and then recover $v_{r,k}$ and $v_{z,k}$ for the axially symmetric Biot-Savart law.

After the Fourier zero and non-zero mode of the velocity $\bl{v}$ are represented by solving the corresponding ODE,  suitable estimates for the solution of the linear system are then given. In this process, we need to apply some basic properties of the modified Bessel function, which reveal that the non-zero mode of the velocity shares the same angular regularity as the boundary value and exhibits the same decay rate as the external force. Since the kernel of the ODE for the zero mode is of Euler type, the decay rate of the zero mode is two orders weaker than that of the external force. This explains why we impose the decay conditions  $\la_\th>3$ for the swirl component of the external force $f_{\th,0}$ and $\la_z>2$ for the vortical external force $f_{z,0}$. Once appropriate estimates of the linear system are established, the standard technique of contraction mapping can be applied to demonstrate the solvability of the nonlinear system \eqref{ReduNS} within our designated functional spaces, provided the data are sufficiently small. In order to close the non-linear estimates, the decay rate of the non-zero mode of the external force must be set to at least $\la>3/2$. However, deriving the estimates for the linear Fourier modes requires tedious and complex computations, and it is the most demanding part in our paper. See Section \ref{sec2} and Section \ref{sec3}.

\subsection{Related works, organization and notations of the paper}

 Now we present some related works. For the two-dimensional exterior system \eqref{2dns}, Hillairet and Wittwer \cite{Hillairet2013} consider the perturbation of system \eqref{2dns} around the scale-critical solution $\f{\mu}{r}\bl{e}_\th$ in an exterior disk for zero flux ($\nu=0$) and zero external force.  When $|\mu|>\s{48}$, the linearized equations can produce a subcritical error solution, then they show the existence of solutions in the form of $\bl{u}=\mu\bl{e}_\th+o(|x|^{-1})$ when $|x|\rightarrow+\i$. Some other references for the flow with zero flux ($\nu=0$), non-perturbed boundary condition and non-zero external force can be found in Higaki \cite{HigakiMN:2018ARMA,GallagherHM:2019MATHNA}.

  Recently Higaki \cite{Higaki2023} consider the external force perturbation effect of system \eqref{2dns} with $\nu<-2$, in which the boundary condition is not perturbed. Later \cite{GH2024} and \cite{LP2024} give improvements to the above-mentioned results by allowing the scale-critical flow  $\f{\nu}{r}\bl{e}_r+\f{\mu}{r}\bl{e}_\th$ that produce a stabilizing effect to the spatial decay when $\mu,\nu$ satisfy some suitable constraints. Also the non-uniqueness result are given in the above two paper \cite{GH2024} and \cite{LP2024} for the case $\nu<-2$.

Study on the axially symmetric Navier-Stokes equations have always been a hot topic. For the unsteady axially symmetric Navier-Stokes equations in $\mathbb{R}^3$, the regularity problem of solutions with a swirl component has drawn significant attention. This is particularly true since Ladyzhenskaya \cite{Lady:1968} and Ukhovskii-Itdovich \cite{UI:1968} independently proved in 1968 that weak solutions are regular for all time in the absence of swirl. Now it is still an open problem. See some process in the last 20 years in \cite{ChenSYT:2008IMRN,ChenSYT:2009CPDE,KochNSS:2009ACTAMATH,LeiZ:2011JFA,Pan:2016JDE,Wei:2016JMAA,ChenFZ:2017DCDS,LeiZ:2017PJM,LiPYZZZ:2024JFA} and references therein. For the stationary axially symmetric NS, many literatures has been devoted in studying the Liouville theorems and asymptotic behavior of the solution at spacial infinity in various domains. See for instance \cite{Wang:2019JDE,SereginW:2019AIA,CarrilloPZ:2020JFA,CarrilloPZZ:2020ARMA,PanL:2020NARWA,KozonoTW:2022JFA,KozonoTW:2023JDE,LiP:2023CCM,LiPY:2025CVPDE}.

Our paper is organized as follows. In Section \ref{sec2}, we formulate the linearized system of \eqref{ReduNS} and deduce its each Fourier mode. Then we solve and estimate the zero mode directly from the equations. In Section \ref{sec3}, we solve and estimate the non-zero modes for $v_\th$ directly and for $(v_r,v_z)$ by applying the equations of the stream function and vorticity. At last, by estimate the nonlinear term and using contract mapping principle, we show the existence of solutions to the nonlinear system in Section \ref{sec4}.

Throughout the paper, $C_{a,b,...}$ denotes a positive constant depending on $a,\,b,\,...$, which may be different from line to line. $A\lesssim_{a,b,...} B$ means $A\leq C_{a,b,...}B$. Finally, $A\thickapprox_{a,b,...} B$ means both $A\lesssim_{a,b,...} B$ and $B\lesssim_{a,b,...} A$ are satisfied.
\section{Reformulation in the Fourier mode and the zero-mode estimates}\label{sec2}
In this section, we consider the linearized system of the reduced axially symmetric Navier-Stokes system \eqref{ReduNS} and formulate it into the Fourier mode. Then the solvability and estimation of the zero mode given. The linearized system for \eqref{ReduNS} is given by
\be\label{PoNSvL00}
\left\{\begin{split}
&-\left(\p_r^2+\f{1-{\nu}}{r}\p_r-\f{1-{\nu}}{r^2}+\p_z^2\right)v_r+\p_r\pi=\bar{f}_r\,,\\[1mm]
&-\left(\p_r^2+\f{1-{\nu}}{r}\p_r-\f{1+{\nu}}{r^2}+\p_z^2\right)v_\th=\bar{f}_\th\,,\\[1mm]
&-\left(\p_r^2+\f{1-{\nu}}{r}\p_r+\p_z^2\right)v_z+\p_z\pi=\bar{f}_z\,,\\[1mm]
&r\p_z v_z+\p_r(rv_r)=0\,,\\
&v_r\big|_{r=1}={g}_r(z),\q v_\th\big|_{r=1}= g_\th(z),\q v_z\big|_{r=1}= g_z(z),\q \bl{v}\big|_{r\rightarrow+\i}=0\,.
\end{split}\right.
\ee
Here $\bl{v}:=v_r(r,z)\bl{e}_r+v_\th(r,z)\bl{e}_\th+v_z(r,z)\bl{e}_z$ is the linearized axisymmetric velocity field. To recover the nonlinear system \eqref{ReduNS}, we just choose
\[
\bali
&\bar{f}_r=-\left({v}_r\p_r+{{v}_z}\p_z\right){v}_r+\f{{v}_\th^2}{r}+{\f{2\mu}{r^2}{v}_\th}+f_r,\\
&\bar{f}_\th=-\left({v}_r\p_r+{{v}_z}\p_z\right){v}_\th-\f{{v}_r{v}_\th}{r}+f_\th,\\
&\bar{f}_z=-\left({v}_r\p_r+{{v}_z}\p_z\right){v}_z+f_z.
\eali
\]

Applying the Fourier series technique, i.e.
\[
v_r(r,z):=\sum_{k\in\mathbb{Z}}v_{r,k}(r)e^{ikz}\,,\q\text{where}\q v_{r,k}(r)=\f{1}{2\pi}\int_0^{2\pi}v_r(r,\th)e^{-ikz}\md z\,,
\]
and similarly for $v_\th$, $v_z$ and the other functions, we rewrite \eqref{PoNSvL00} in each $k$-mode for any $k\in\mathbb{Z}$ :
\be\label{elinear}
\left\{\begin{aligned}
&-\left(\f{\md^2}{\md r^2}+\f{1-{\nu}}{r}\f{\md}{\md r}-\f{1-{\nu}}{r^2}-k^2\right)v_{r,k}+\f{\md}{\md r}\pi_k=\bar{f}_{r,k}\,,&\,\text{for}\q r\in]1,\i[\,,\\[1mm]
&-\left(\f{\md^2}{\md r^2}+\f{1-{\nu}}{r}\f{\md}{\md r}-\f{1+\nu}{r^2}-k^2\right)v_{\th,k}=\bar{f}_{\th,k}\,,&\text{for}\q r\in]1,\i[\,,\\[1mm]
&-\left(\f{\md^2}{\md r^2}+\f{1-{\nu}}{r}\f{\md}{\md r}-k^2\right)v_{z,k}+ik\pi_k=\bar{f}_{z,k}\,,&\text{for}\q r\in]1,\i[\,,\\[1mm]
&ikv_{z,k}+\f{\md}{\md r}v_{r,k}+\f{v_{r,k}}{r}=0\,,&\text{for}\q r\in]1,\i[\,,\\[1mm]
&v_{r,k}(1)=g_{r,k}\,,\q v_{\th,k}(1)=g_{\th,k}\,\q v_z(1)=g_{z,k}\,,\q v_{r,k}(+\i)=v_{\th,k}(+\i)&=v_{z,k}(+\i)=0.
\end{aligned}\right.
\ee

In the following, we first consider the zero mode for $k=0$, since there is no pressure term, which is not hard to solve. From \eqref{elinear}$_{2,3,4,5}$ for $k=0$, we see that
\be\label{elinearzeromode}
\left\{\begin{aligned}
&-\left(\f{\md^2}{\md r^2}+\f{1-{\nu}}{r}\f{\md}{\md r}-\f{1+\nu}{r^2}\right)v_{\th,0}=\bar{f}_{\th,0}\,,&\q\text{for}\q r\in]1,\i[\,,\\[1mm]
&-\left(\f{\md^2}{\md r^2}+\f{1-{\nu}}{r}\f{\md}{\md r}\right)v_{z,0}=\bar{f}_{z,0}\,,&\q\text{for}\q r\in]1,\i[\,,\\[1mm]
&\f{\md}{\md r}v_{r,0}+\f{v_{r,0}}{r}=0\,,&\q\text{for}\q r\in]1,\i[\,,\\[1mm]
&(v_{r,0},v_{\th,0},v_{z,0})(1)=(g_{r,0},g_{\th,0},g_{z,0})\,,\q (v_{r,0},v_{\th,0},v_{z,0})(+\i)=0.&
\end{aligned}\right.
\ee

\subsection{The zero mode of $\bl{v_\th}$ and its estimate}

The zero mode $v_{\th,0}$ satisfies the following ODE,
\be\label{PoNSvLF0}
\left\{\begin{split}
&-\left(\f{\md^2}{\md r^2}+\f{1-\nu}{r}\f{\md}{\md r}-\f{1+\nu}{r^2}\right)v_{\th,0}=\bar{f}_{\th,0}\,,\q\q\q\q\q\text{for}\q r\in]1,\i[\,,\\[1mm]
&v_{\th,0}(1)=g_{\th,0}\,\q v_{\th,0}(+\i)=0.
\end{split}\right.
\ee
This is a boundary value problem of Eulerian ODE. We have the following proposition for it.
\begin{proposition}\label{Prop1}
Given $\bar{f}_{\th,0}\in L^\i_{\la_\th}$ with $\la_\th>3$. The boundary value problem \eqref{PoNSvLF0} has a unique solution with the form
\[
v_{\th,0}(r)=\mathrm{v}(r)+{\bf1}_{-2\leq\nu<0}\f{\sigma}{r}\,,
\]
where $\sigma\in\bR$ and
\[
\mathrm{v}(r)=o(r^{-1})\,\q\text{as}\q r\to\i\,.
\]
And it that satisfies the estimate
\be\label{Mainvth0}
\|{\mathrm{v}}''\|_{L^\i_{\tilde{\la_\th}}}+\|{\mathrm{v}}'\|_{L^\i_{\tilde{\la_\th}-1}}+\|{\mathrm{v}}\|_{L^\i_{\tilde{\la_\th}-2}}+{\bf1}_{-2\leq\nu<0}|\sigma|\les_{\nu,\t{\la_\th}}\|\bar{f}_{\th,0}\|_{L^\i_{\t{\la_\th}}}+|g_{\th,0}|\,.
\ee
Here $\t{\la_\th}\in\,]3,\bar{\la_\th}]$ , where
\be\label{lath}
\bar{\la_\th}:=\left\{
\begin{aligned}
&\min\left\{\la_\th\,,\,2-\f{\nu}{2}\right\},\q&\text{for}\q \nu<-2\,;\\
&\la_\th\q&\text{for}\q \nu\geq-2\,.\\
\end{aligned}
\right.
\ee
\end{proposition}
\begin{proof}
Two linearly independent solutions of the homogeneous equation
\be\label{Homth}
-\left(\f{\md^2}{\md r^2}+\f{1-\nu}{r}\f{\md}{\md r}-\f{1+\nu}{r^2}\right)\tilde{v}_{\th,0}=0
\ee
are $r^{1+\nu}$ and $r^{-1}$.

\noindent {\bf\noindent The case of $\bl{\nu<-2}$.}

The ODE \eqref{PoNSvLF0} with boundary condition $v_{\th,0}(1)=g_{\th,0}$ has the following solution decaying as $O(r^{-1^{+}})$ at spacial infinity:
\be\label{v0<}
\begin{aligned}
v_{\theta, 0}(r)=&-\frac{1}{\nu+2}\left\{r^{\nu+1} \int_1^r s^{-\nu} \bar{f}_{\theta, 0}(s) \mathrm{d} s+r^{-1} \int_r^{\infty} s^2 \bar{f}_{\theta, 0}(s) \mathrm{d} s\right\}\\
                 &+\left(g_{\th,0}+\frac{1}{\nu+2}\int_1^{\infty} s^2 \bar{f}_{\theta, 0}(s) \mathrm{d} s\right)r^{\nu+1}\,.
\end{aligned}
\ee
Recall \eqref{lath}, clearly it satisfies
\[
\tilde{\la_\th}\leq\la_\th,\q\text{and}\q\tilde{\la_\th}\in]3,1-\nu[\,.
\]
Direct calculation from \eqref{v0<} indicates
\be\label{Evth0}
\begin{split}
|r^{{\tilde{\la_\th}}-2}v_{\th,0}(r)|&\leq-\f{\|\bar{f}_{\th,0}\|_{L^\i_{\tilde{\la_\th}}}}{\nu+2}\left(r^{\tilde{\la_\th}+\nu-1}\int_1^rs^{-\nu-\tilde{\la_\th}}\mathrm{d} s+r^{\tilde{\la_\th}-3} \int_r^{\infty} s^{2-\tilde{\la_\th}}\mathrm{d} s+r^{\tilde{\la_\th}+\nu-1}\int_1^\i s^{2-\tilde{\la_\th}}\mathrm{d}s\right)\\[2mm]
&\hskip .5cm+|g_{\th,0}|r^{\tilde{\la_\th}+\nu-1}\\[2mm]
&=-\f{\|\bar{f}_{\th,0}\|_{L^\i_{\tilde{\la_\th}}}}{\nu+2}\left(\f{1}{-\nu-\tilde{\la_\th}+1}(1-r^{\nu+\tilde{\la_\th}-1})+\f{1}{\tilde{\la_\th}-3}(1+r^{\tilde{\la_\th}+\nu-1})\right)+|g_{\th,0}|r^{\nu+\tilde{\la_\th}-1}\\
&\les_{\nu,\tilde{\la_\th}}\|\bar{f}_{\th,0}\|_{L^\i_{{\tilde{\la_\th}}}}+|g_{\th,0}|\,.
\end{split}
\ee
Meanwhile, direct calculation shows
\[
\begin{split}
\f{\md}{\md r}v_{\theta, 0}(r)=&-\f{1}{\nu+2}\left\{(\nu+1)r^\nu\int_1^r s^{-\nu}\bar{f}_{\th,0}(s)\md s-r^{-2}\int_r^\i s^2\bar{f}_{\th,0}(s)\md s\right\}\\
&+(\nu+1)\left(g_{\th,0}+\frac{1}{\nu+2}\int_1^{\infty} s^2 \bar{f}_{\theta, 0}(s) \mathrm{d} s\right)r^{\nu}\,.
\end{split}
\]
And similarly as one derives \eqref{Evth0} and using \eqref{PoNSvLF0}, we have
\be\label{vthe0first}
|r^{\tilde{\la_\th}-1}v'_{\th,0}(r)|+|r^{\tilde{\la_\th}}{v}''_{\th,0}(r)|\les_{\nu,\tilde{\la_\th}}\|\bar{f}_{\th,0}\|_{L^\i_{{\tilde{\la_\th}}}}+|g_{\th,0}|\,.
\ee
Combining \eqref{Evth0} and \eqref{vthe0first}, we see that
\be\label{vthe0first00}
\|{v}''_{\th,0}\|_{L^\i_{\tilde{\la_\th}}}+\|v'_{\th,0}\|_{L^\i_{\tilde{\la_\th}-1}}+\|v_{\th,0}\|_{L^\i_{\tilde{\la_\th}-2}}\les_{\nu,\tilde{\la_\th}}\|\bar{f}_{\th,0}\|_{L^\i_{{\tilde{\la_\th}}}}+|g_{\th,0}|\,.
\ee
This completes the estimates of $v_{\th,0}$ when $\nu<-2$.

{\bf\noindent The case of $\bl{-2\leq\nu<0}$.}

 The function $r^{\nu+1}$ (or $\log r$ when $\nu=-2$) decays slower than the prescribed request, one cannot solve the ODE of $v_{\th,0}$ to get a solution decays faster than $r^{-1}$ for any given $g_{\th,0}$. Instead, one has the following exact solution of \eqref{PoNSvLF0}
\be\label{v0>}
{v}_{\th,0}(r)=-\f{1}{r}\int_r^\i s^{\nu+1}\int_s^\i t^{-\nu}\bar{f}_{\th,0}(t)\mathrm{d}t\mathrm{d}s+\f{\sigma}{r}:=\mathrm{v}(r)+\f{\sigma}{r},
\ee
where
\be\label{vth0thi001}
\sigma=\int_1^\i s^{\nu+1}\int_s^\i t^{-\nu}\bar{f}_{\th,0}(t)\mathrm{d}t\mathrm{d}s+g_{\th,0}\,.
\ee
Owing to $\bar{f}_{\th,0}\in L^\i_{\la_\th}$, we infer from \eqref{v0>} that
\be\label{vth0sec}
\begin{split}
|r^{\tilde{\la_\th}-2}\mathrm{v}(r)|&\leq r^{\tilde{\la_\th}-3}\|\bar{f}_{\th,0}\|_{L^\i_{\tilde{\la_\th}}}\int_r^\i s^{\nu+1}\int_s^\i t^{-\nu-\tilde{\la_\th}}\mathrm{d}t\mathrm{d}s\leq\f{1}{(\nu+\tilde{\la_\th}-1)(\tilde{\la_\th}-3)}\|\bar{f}_{\th,0}\|_{L^\i_{{\tilde{\la_\th}}}}\,.\\
\end{split}
\ee
Meanwhile, since
\[
\f{\md}{\md r}\mathrm{v}(r)=\f{1}{r^2}\int_r^\i s^{\nu+1}\int_s^\i t^{-\nu}\bar{f}_{\th,0}(t)\mathrm{d}t\mathrm{d}s+r^\nu\int_r^\i t^{-\nu}\bar{f}_{\th,0}(t)\mathrm{d}t\,,
\]
one derives
\be\label{vth0thi}
|r^{\tilde{\la_\th}-1}\mathrm{v}'(r)|\les_{\nu,\tilde{\la_\th}}\|\bar{f}_{\th,0}\|_{L^\i_{{\tilde{\la_\th}}}}\,.
\ee
similarly as achieving the estimate \eqref{vth0sec}. Noticing that $\f{\sigma}{r}$ is an exact solution of \eqref{Homth}, one deduces from \eqref{PoNSvLF0}$_1$ that:
\be\label{Evth0123}
|r^{\tilde{\la_\th}}\mathrm{v}''|\les_\nu|r^{\tilde{\la_\th}-1}\mathrm{v}'|+|r^{\tilde{\la_\th}-2}\mathrm{v}|+|r^{\tilde{\la_\th}}\bar{f}_{\th,0}|\,.
\ee
Combining \eqref{vth0thi001}, \eqref{vth0sec}, \eqref{vth0thi} and \eqref{Evth0123}, we can obtain that
\be\label{vth0firsta}
\|{\mathrm{v}}''\|_{L^\i_{\tilde{\la_\th}}}+\|{\mathrm{v}}'\|_{L^\i_{\tilde{\la_\th}-1}}+\|{\mathrm{v}}\|_{L^\i_{\tilde{\la_\th}-2}}+|\sigma|\les_{\nu,\tilde{\la_\th}}\|\bar{f}_{\th,0}\|_{L^\i_{{\tilde{\la_\th}}}}+|g_{\th,0}|\,.
\ee
Combining \eqref{vthe0first00} and \eqref{vth0firsta}, we conclude \eqref{Mainvth0}. This completes the proof of the proposition.
\end{proof}

\subsection{The zero modes of $\bl{v_r}$ and $\bl{v_z}$ and their estimates}
For the zero modes of $v_r$ and $v_z$, we go back to the velocity equations \eqref{elinearzeromode}$_{2,3}$ to get
\be\label{vrz0}
\left\{\begin{aligned}
&-\left(\f{\md^2}{\md r^2}+\f{1-{\nu}}{r}\f{\md}{\md r}\right)v_{z,0}=\bar{f}_{z,0}\,,\q&\text{for}\q r\in]1,\i[\,,\\[1mm]
&\f{\md}{\md r}v_{r,0}+\f{v_{r,0}}{r}=0\,,\q&\text{for}\q r\in]1,\i[\,,\\[1mm]
&v_{r,0}(1)=g_{r,0}=0\,\q v_{z,0}(1)=g_{z,0}\,,\q v_{r,0}(+\i)=v_{z,0}(+\i)=0\,.&
\end{aligned}\right.
\ee
Here is the main result of the subsection:
\begin{proposition}\label{Prop4}
Given $\bar{f}_{z,0}\in L^\i_{\la_z}$ with $\la_z>2$. The boundary value problem \eqref{vrz0} has the unique solution:
\be\label{svrz0}
\bali
&v_{r,0}(r)=0\,\\
&v_{z,0}(r)=\left(g_{z,0}-\nu^{-1}\int_1^\i \bar{f}_{z,0}(s)s\md s\right)r^\nu+\nu^{-1}r^\nu\int_1^r\bar{f}_{z,0}(s)s^{-\nu+1}\md s+\nu^{-1}\int_r^\i \bar{f}_{z,0}(s)s\md s\,.
\eali
\ee
And $v_{z,0}$ satisfies the estimate
\be\label{vze0firsta}
\|{v}''_{z,0}\|_{L^\i_{\t{\la_z}}}+\|v'_{z,0}\|_{L^\i_{\t{\la_z}-1}}+\|v_{z,0}\|_{L^\i_{\t{\la_z}-2}}\les_{\nu,\t{\la_z}}\|\bar{f}_{z,0}\|_{L^\i_{\t{\la_z}}}+|g_{z,0}|\,.
\ee
for any $\t{\la_z}\in\,]2,\bar{\la_z}]$ . Here
\be\label{laz}
\bar{\la_z}=\min\left\{\la_z\,,\,2-\f{\nu}{2}\right\}\,.
\ee
\end{proposition}
\begin{proof}
The equation \eqref{vrz0}$_2$ indicates
\[
v_{r,0}=\f{C}{r}\,.
\]
And the constant $C$ must be zero owing to the boundary condition.
Moreover, akin to \eqref{v0<}, the function $v_{z,0}$ is solved as
\be\label{Evz0}
v_{z,0}(r)=\bar{v}_0r^\nu+\nu^{-1}r^\nu\int_1^r\bar{f}_{z,0}(s)s^{-\nu+1}\md s+\nu^{-1}\int_r^\i \bar{f}_{z,0}(s)s\md s\,.
\ee
Recall the boundary condition in \eqref{vrz0}$_3$, we have
\[
g_{z,0}=\bar{v}_0+\nu^{-1}\int_1^\i \bar{f}_{z,0}(s)s\md s\,,
\]
which indicates
\[
\bar{v}_0=g_{z,0}-\nu^{-1}\int_1^\i \bar{f}_{z,0}(s)s\md s\,.
\]
This proves \eqref{svrz0}. Noticing that $\la_z>2$, direct calculation shows
\[
|\bar{v}_0|\les_{\nu,{\tilde{\la_z}}}|g_{z,0}|+\|\bar{f}_{z,0}\|_{L^\i_{{\tilde{\la_z}}}}\,.
\]
Recalling \eqref{laz}, it is clear that $\tilde{\la_z}\in]2,2-\nu[$ . Direct calculation shows
\be\label{Evz01}
\begin{split}
|r^{\tilde{\la_z}-2}v_{z,0}(r)|&\leq-\f{\|\bar{f}_{z,0}\|_{L^\i_{\tilde{\la_z}}}}{\nu}\left(r^{\tilde{\la_z}+\nu-2}\int_1^rs^{-\nu+1-\tilde{\la_z}}\mathrm{d} s+r^{\tilde{\la_z}-2} \int_r^{\infty} s^{1-\tilde{\la_z}}\mathrm{d} s\right)+\left(|g_{z,0}|+\|\bar{f}_{z,0}\|_{L^\i_{\tilde{\la_z}}}\right)r^{\tilde{\la_z}+\nu-2}\\[2mm]
&=-\f{\|\bar{f}_{z,0}\|_{L^\i_{\t{\la_z}}}}{\nu}\left(\f{1}{-\nu-\tilde{\la_z}+2}(1-r^{\nu+\tilde{\la_z}-2})+\f{1}{\tilde{\la_z}-2}\right)+\left(|g_{z,0}|+\|\bar{f}_{z,0}\|_{L^\i_{\t{\la_z}}}\right)r^{\tilde{\la_z}+\nu-2}\\
&\les_{\nu,{\tilde{\la_z}}}\|\bar{f}_{z,0}\|_{L^\i_{{{\tilde{\la_z}}}}}+|g_{z,0}|\,.
\end{split}
\ee
Meanwhile, taking the derivative of \eqref{Evz0}, one deduces
\[
\begin{split}
\f{\md}{\md r}v_{z, 0}(r)=&\nu\bar{v}_0r^{\nu-1}+r^{\nu-1}\int_1^r\bar{f}_{z,0}(s)s^{-\nu+1}\md s\,.
\end{split}
\]
And using a similarly method as one derives \eqref{Evz01} and using \eqref{vrz0}, we conclude that
\be\label{vze0first}
|r^{{\tilde{\la_z}}-1}v'_{z,0}(r)|+|r^{{\tilde{\la_z}}}{v}''_{z,0}(r)|\les_{\nu,{\tilde{\la_z}}}\|\bar{f}_{z,0}\|_{L^\i_{{\tilde{\la_z}}}}+|g_{z,0}|\,.
\ee
Combining \eqref{Evz01} and \eqref{vze0first}, we derive \eqref{vze0firsta}. This completes the proof of Proposition \ref{Prop4}.
\end{proof}

\section{The non-zero Fourier modes and their estimates}\label{sec3}

\subsection{The non-zero mode of $\bl{v_\th}$  and its estimate}
From \eqref{elinear}$_2$, in this subsection, we consider the boundary value problem of $k$ mode of $v_\th$ for $k\neq 0$.
\be\label{PoNSvLFk}
\left\{\begin{split}
&-\left(\frac{\md^2}{\md r^2}+\frac{1-\nu}{r} \frac{\md}{\md r}-\frac{1+\nu}{r^2}-k^2\right) v_{\th,k}=\bar{f}_{\th,k}\,,\q\q\q\q\q\text{for}\q r\in(1,\i)\,,\\[1mm]
&v_{\th,k}(1)=g_{\th,k}\,\q v_{\th,k}(+\i)=0.
\end{split}\right.
\ee
 We have the following result:
\begin{proposition}\label{Prop2}
For $k\in\mathbb{Z}-\{0\}$, given $\bar{f}_{\th,k}\in L^\i_{\la}$ with $\la>3/2$. The boundary value problem \eqref{PoNSvLFk} has a unique solution
\[
v_{\th,k}(r)=o(r^{-1})\,\q\text{as}\q r\to\i\,.
\]
And it satisfies the estimate
\be\label{Mainvthk}
\|v_{\th,k}''\|_{L^\i_{\t{\la}}}+|k|\cd\|v_{\th,k}'\|_{L^\i_{\t{\la}}}+k^2\|v_{\th,k}\|_{L^\i_{\t{\la}}}\les_{\nu,\t{\la}}\|\bar{f}_{\th,k}\|_{L^\i_{\t{\la}}}+k^2|g_{\th,k}|\,
\ee
for any $\t{\la}\in\,]\f{3}{2},\la]$ .
\end{proposition}

Different from \eqref{Homth}, the homogeneous equation of \eqref{PoNSvLFk}$_1$:
\be\label{BE1}
-\left(\frac{\md^2}{\md r^2}+\frac{1-\nu}{r} \frac{\md}{\md r}-\frac{1+\nu}{r^2}-k^2\right) \tilde{v}_{\th,k}(r)=0
\ee
with $k\neq 0$ is no longer of the Eulerian type. Nevertheless, by rewriting
\be\label{Cgv}
\tilde{v}_{\th,k}(r)=r^{\f{\nu}{2}}g(|k|r)\,,
\ee
direct calculation shows $g$ satisfies
\[
-g''-\f{1}{r}g'+\Big(1+\f{\left(1+\f{\nu}{2}\right)^2}{r^2}\Big)g=0\,.
\]
This is \emph{the modified Bessel equation} which has two linearly independent solutions
\[
\left\{
\begin{split}
&I_{\left|1+\f{\nu}{2}\right|}(r)=\sum_{m=0}^{\infty} \frac{1}{m!\Gamma\left(m+\left|1+\f{\nu}{2}\right|+1\right)}\left(\frac{r}{2}\right)^{2 m+\left|1+\f{\nu}{2}\right|}\,;\\[2mm]
&K_{\left|1+\f{\nu}{2}\right|}(r)=\frac{\pi}{2} \frac{I_{-\left|1+\f{\nu}{2}\right|}(r)-I_{\left|1+\f{\nu}{2}\right|}(r)}{\sin \left(\left|1+\f{\nu}{2}\right| \pi\right)}\,.
\end{split}
\right.
\]
Here, $I_{\left|1+\f{\nu}{2}\right|}(r)$ and $K_{\left|1+\f{\nu}{2}\right|}(r)$ are modified Bessel functions of the first and the second kind with index $\left|1+\f{\nu}{2}\right|$, respectively. They  are positive for all $r>0$. Meanwhile, for $r\geq 1 $,
\be\label{bessel}
I_\al(r) \thickapprox_\al r^{-\f{1}{2}r}e^r,\q \text{ and }\q K_\al(r) \thickapprox_\al r^{-\f{1}{2}}e^{-r}.
\ee
We refer readers to \cite[page 377-378]{AS1992} for details. Recall \eqref{Cgv}, we denote
\be\label{FUnc}
\mathcal{K}_k(r):=r^{\f{\nu}{2}}K_{\left|1+\f{\nu}{2}\right|}(|k|r)\,,\q \mathcal{I}_k(r):=r^{\f{\nu}{2}}I_{\left|1+\f{\nu}{2}\right|}(|k|r)\,,
\ee
for the convenience.

To prove Proposition \ref{Prop2}, the following basic estimates of the prescribed function $\mK_k$ and $\mI_k$ are required:
\begin{lemma}\label{lemKI}
The linearly independent fundamental solutions $\mK_k(r)$ and $\mI_k(r)$ satisfy the following estimates
\begin{align}
&\f{\md^n}{\md r^n}\mK_k(r)\les_\nu |k|^{-\f{1}{2}+n}r^{\f{\nu-1}{2}}e^{-|k|r}, \q \f{\md^n}{\md r^n}\mI_k(r)\les_\nu|k|^{-\f{1}{2}+n}r^{\f{\nu-1}{2}}e^{|k|r} \label{bessel1}\,,
\end{align}
for $n=0,\,1,\,2$.
\end{lemma}
\begin{proof}
The case of $n=0$ in \eqref{bessel1} is a direct consequence of \eqref{bessel}. For $n=1$, by using the property of the modified Bessel function
\be\label{BP111}
\f{\md}{\md x}K_\al(x)=\f{\al}{x}K_\al(x)-K_{1+\al}(x)\,,
\ee
 (see \cite[page 376]{AS1992}) for instants.) we first have
\begin{align}
\mK^{\prime}_k(r)=&\f{\nu}{2}r^{\f{\nu}{2}-1} K_{|1+\f{\nu}{2}|}(|k|r)+|k|r^{\f{\nu}{2}}K^\prime_{|1+\f{\nu}{2}|}(|k|r)\nn\\[1mm]
=&\f{\nu}{2}r^{\f{\nu}{2}-1} K_{|1+\f{\nu}{2}|}(|k|r)+|k|r^{\f{\nu}{2}}\lt( \f{|1+\f{\nu}{2}|}{|k|r}K_{|1+\f{\nu}{2}|}(|k|r)-K_{1+|1+\f{\nu}{2}|}(|k|r) \rt)\nn\\[1mm]
=&\lt(\f{\nu}{2}+|1+\f{\nu}{2}|\rt)r^{\f{\nu}{2}-1} K_{|1+\f{\nu}{2}|}(|k|r)-|k|r^{\f{\nu}{2}}K_{1+|1+\f{\nu}{2}|}(|k|r)\,.\label{bessel5}
\end{align}
Then again using \eqref{bessel}, we see that
\[
|\mK^{\prime}_k(r)|\ls_\nu |k| r^{\f{\nu}{2}} (|k|r)^{-\f{1}{2}}e^{-|k|r}\ls_\nu |k|^{\f{1}{2}}r^{\f{\nu-1}{2}}e^{-|k|r}.
\]
Taking derivative of \eqref{bessel5} again and using \eqref{BP111} and \eqref{bessel}, we can also obtain
\[
|\mK^{\prime\prime}_k(r)|\ls_\nu |k|^2 r^{\f{\nu}{2}} (|k|r)^{-\f{1}{2}}e^{-|k|r}\ls_\nu |k|^{\f{3}{2}}r^{\f{\nu-1}{2}}e^{-|k|r}.
\]
This proves the first estimate in \eqref{bessel1}. Also, by using (see \cite[page 376]{AS1992}) for instants.)
 \be\label{BP222}
 \f{\md}{\md x}I_\al(x)=\f{\al}{x}I_\al(x)+I_{1+\al}(x)\,,
 \ee
 we derive that
\begin{align}
\mI^{\prime}_k(r)=&\f{\nu}{2}r^{\f{\nu}{2}-1} I_{|1+\f{\nu}{2}|}(|k|r)+|k|r^{\f{\nu}{2}}I^\prime_{|1+\f{\nu}{2}|}(|k|r)\nn\\
=&\f{\nu}{2}r^{\f{\nu}{2}-1} I_{|1+\f{\nu}{2}|}(|k|r)+|k|r^{\f{\nu}{2}}\lt( \f{|1+\f{\nu}{2}|}{|k|r}I_{|1+\f{\nu}{2}|}(|k|r)+I_{1+|1+\f{\nu}{2}|}(|k|r) \rt)\nn\\
=&\lt(\f{\nu}{2}+|1+\f{\nu}{2}|\rt)r^{\f{\nu}{2}-1} I_{|1+\f{\nu}{2}|}(|k|r)+|k|r^{\f{\nu}{2}}I_{1+|1+\f{\nu}{2}|}(|k|r)\,. \label{bessel6}
\end{align}
Then again using \eqref{bessel}, we see that
\[
|\mI^{\prime}_k(r)|\ls_\nu |k| r^{\f{\nu}{2}} (|k|r)^{-\f{1}{2}}e^{|k|r}\ls_\nu |k|^{\f{1}{2}}r^{\f{\nu-1}{2}}e^{|k|r}\,.
\]
Taking derivative of \eqref{bessel6} again and using \eqref{BP222} and \eqref{bessel}, we can similarly obtain
\[
|\mI^{\prime\prime}_k(r)|\ls_\nu |k|^2 r^{\f{\nu}{2}} (|k|r)^{-\f{1}{2}}e^{|k|r}\ls_\nu |k|^{\f{3}{2}}r^{\f{\nu-1}{2}}e^{|k|r}\,.
\]
This finishes the proof of the lemma.

\end{proof}

Since we intend to solve a subcritically decayed solution such that insist $v_{\th,k}(r) = o(r^{-1})$ as $r\to\i$, the general solution of \eqref{elinear}$_1$ can be represented by for $k\neq 0$,
\be\label{subn0}
\begin{split}
v_{\th,k}(r)=&\bar{v}_{k} \mK_k(r)+\mK_k(r)\int_1^r \bar{f}_{\th,k}(s)s^{1-\nu}\mI_k(s)\md s+\mI_k(r)\int_r^{\i}\bar{f}_{\th,k}(s)s^{1-\nu}\mK_k(s)\md s.
\end{split}
\ee
Recall the boundary condition \eqref{elinear}$_4$, we arrive at
\[
g_{\th,k}=v_{\th,k}(1)=\bar{v}_k\mK_k(1)+I_k(1)\int_1^\i \bar{f}_{\th,k} (s)s^{1-\nu}\mK_k(s)\md s\,.
\]
Thus one can solve the constant $\bar{v}_k$:
\be\label{linearvth1}
\begin{aligned}
\bar{v}_{k} & =\frac{1}{\mK_k(1)}\left(g_{\theta, k}-\mI_k(1) \int_1^{\infty} \bar{f}_{\th,k} (s) s^{1-\nu} \mK_k(s) \md s\right) \\
& =\frac{1}{K_{\left|1+\frac{v}{2}\right|}(|k|)}\left(g_{\theta, k}-I_{\left|1+\frac{v}{2}\right|}(|k|) \int_1^{\infty} \bar{f}_{\th,k} (s) s^{1-\nu} \mK_k(s) \md s\right)\,.
\end{aligned}
\ee
Substituting \eqref{linearvth1} in \eqref{subn0}, one concludes tha for $k\in\bZ-\{0\}$,
\be\label{linearvth}
\begin{aligned}
v_{\th, k}(r)=&\frac{\mK_k(r) }{K_{\left|1+\frac{\nu}{2}\right|}(|k|)}\left(g_{\theta, k}-I_{\left|1+\frac{\nu}{2}\right|}(|k|) \int_1^{\infty} \bar{f}_{\th,k} (s) s^{1-\nu} \mK_k(s) \md s\right) \\
& +\mK_k(r)\int_1^r\bar{f}_{\th,k} (s)s^{1-\nu}\mI_k(s)\md s+\mI_k(r)\int_r^{\i}\bar{f}_{\th,k} (s)s^{1-\nu}\mK_k(s)\md s.
\end{aligned}
\ee

{
For the further estimates of the $v_{\th,k}$, the following estimates are needed.
\begin{lemma}
For any $\al\in \bR$, $r>1$ and $k\in\bZ-\{0\}$, it holds that
\be\label{bessel2}
\int^r_1 s^\al e^{|k|s}\md s\ls_\al |k|^{-1} e^{|k|r} r^\al,\q \int^{+\i}_r s^\al e^{-|k|s}ds\ls_\al |k|^{-1} e^{-|k|r} r^\al.
\ee
\end{lemma}
\begin{proof}
Direct calculation shows
\[
\int^r_1 s^\al e^{|k|s}\md s=\int^r_1 s^\al e^{\f{s}{2}}e^{\left(|k|-\f{1}{2}\right)s}\md s\les_{\al}r^\al e^{\f{r}{2}}\int_1^r e^{\left(|k|-\f{1}{2}\right)s}\md s=\left(|k|-\f{1}{2}\right)^{-1}r^\al e^{\f{r}{2}}\left(e^{\left(|k|-\f{1}{2}\right)r}-e^{|k|-\f{1}{2}}\right)\,.
\]
This proves the first estimate in \eqref{bessel2}. Similarly,
\[
\int^\i_r s^\al e^{-|k|s}\md s=\int^\i_r s^\al e^{-\f{s}{2}}e^{\left(\f{1}{2}-|k|\right)s}\md s\les_{\al}r^\al e^{-\f{r}{2}}\int_r^\i e^{\left(\f{1}{2}-|k|\right)s}\md s=\left(|k|-\f{1}{2}\right)^{-1}r^\al e^{-\f{r}{2}}e^{\left(\f{1}{2}-|k|\right)r}\,,
\]
which shows the validity of the second estimate.
\end{proof}

\leftline{\emph{Proof of Proposition \ref{Prop2}.}} Now it remains to bound the theta component of the velocity given in \eqref{subn0}. First we estimate the coefficient $\bar{v}_{k}$ given in \eqref{linearvth1}. Direct calculation shows
\begin{align}
|\bar{v}_{k}| &\ls |k|^{\f{1}{2}} e^{|k|}\left(|g_{\theta, k}|+|k|^{-\f{1}{2}}  e^{|k|}\lt|\int_1^{\infty} \bar{f}_{\th,k} (s) s^{1-\nu} \mK_k(s) \md s\rt|\right) \nn\\
       &\ls_\nu |k|^{\f{1}{2}} e^{|k|}\left(|g_{\theta, k}|+|k|^{-\f{1}{2}}  e^{|k|}\|\bar{f}_{\th,k}\|_{L^{\i}_{\t{\la}}} \lt|\int_1^{\infty}  s^{1-\nu-\t{\la}} |k|^{-\f{1}{2}}s^{\f{\nu-1}{2}}e^{-|k|s} \md s\rt|\right)\nn\\
        &\ls_{\nu,\t{\la}} |k|^{\f{1}{2}} e^{|k|}\left(|g_{\theta, k}|+|k|^{-2}\|\bar{f}_{\th,k}\|_{L^{\i}_{\t{\la}}} \right). \label{linearvth2}
\end{align}
Here we have applied Lemma \ref{lemKI} in the first and second inequality. And the third line holds because
\[
\lt|\int_1^{\infty}  s^{\f{1}{2}-\f{\nu}{2}-{\tilde{\la}}}e^{-|k|s} \md s\rt|\leq\sup_{s\in[1,\i)}s^{\f{1}{2}-\f{\nu}{2}-{\tilde{\la}}}e^{-\f{s}{2}}\int_1^\i e^{\left(\f{1}{2}-|k|\right)s} \md s\lesssim_{\nu,{\tilde{\la}}}|k|^{-1}e^{-|k|}\,.
\]
Then by using estimates in \eqref{bessel1}, we can obtain that
\begin{align}
k^2|\bar{v}_{k}\mK_k(r)|+|k|\cd|\bar{v}_{k}\mK^\prime_k(r)|\ls r^{\f{\nu-1}{2}} k^2 e^{|k|(1-r)}\left(|g_{\theta, k}|+|k|^{-2}\|\bar{f}_{\th,k}\|_{L^{\i}_{\tilde{\la}}} \right). \label{linearvth2}
\end{align}

Now we give the estimate of integral terms in \eqref{linearvth}. First we consider
\[
\mathcal{L}_1(r):=\mK_k(r)\int_1^r\bar{f}_{\th,k} (s)s^{1-\nu}\mI_k(s)\md s\,.
\]
Using \eqref{bessel1} and \eqref{bessel2}, we have
\begin{align}\label{EEEE0}
|\mathcal{L}_1(r)|& \ls_\nu |k|^{-\f{1}{2}}r^{\f{\nu-1}{2}}e^{-|k|r} \|\bar{f}_{\th,k}\|_{L^{\i}_{\tilde{\la}}}\int^r_1 s^{1-\nu-{\tilde{\la}}}   |k|^{-\f{1}{2}}s^{\f{\nu-1}{2}}e^{|k|s} \md s\nn\\
&\ls_{\nu,{\tilde{\la}}} |k|^{-2}r^{\f{\nu-1}{2}}e^{-|k|r} \|\bar{f}_{\th,k}\|_{L^{\i}_{\tilde{\la}}} r^{\f{1-\nu}{2}-{\tilde{\la}}}e^{|k|r}=|k|^{-2}\|\bar{f}_{\th,k}\|_{L^{\i}_{\tilde{\la}}}r^{-{\tilde{\la}}}\,.
\end{align}
Taking the first derivative of $\mathcal{L}_1(r)$, using \eqref{bessel1} and \eqref{bessel2}, we can see that
\begin{align}
|\mathcal{L}^\prime_1(r)|\leq& \lt|\mK^\prime_k(r)\int_1^r \bar{f}_{\th,k} (s)s^{1-\nu}\mI_k(s)\md s\rt|+\lt|\mK_k(r) \bar{f}_{\th,k} (r)r^{1-\nu}\mI_k(r)\rt|\nn\\[1mm]
                    \ls_{\nu,{\tilde{\la}}}& |k|^{-1}\|\bar{f}_{\th,k}\|_{L^{\i}_{\tilde{\la}}}r^{-{\tilde{\la}}}+ |k|^{-\f{1}{2}}r^{\f{\nu-1}{2}}e^{-|k|r} \|\bar{f}_{\th,k}\|_{L^{\i}_{\tilde{\la}}} r^{-{\tilde{\la}}}r^{1-\nu}|k|^{-\f{1}{2}}r^{\f{\nu-1}{2}}e^{|k|r}\nn\\[1mm]
                     \ls_{\nu,{\tilde{\la}}}& |k|^{-1}\|\bar{f}_{\th,k}\|_{L^{\i}_{\tilde{\la}}}r^{-{\tilde{\la}}}.\label{EEEE}
\end{align}
Then we conclude that
\be\label{linearvth3}
k^2|\mathcal{L}_1(r)|+|k|\cd|\mathcal{L}^\prime_1(r)|\ls_{\nu,{\tilde{\la}}} \|\bar{f}_{\th,k}\|_{L^{\i}_{\tilde{\la}}}r^{-{\tilde{\la}}}
\ee
by combining \eqref{EEEE0} and \eqref{EEEE}. Akin to the above estimates, one denotes
\[
\mathcal{L}_\i(r)=\mI_k(r)\int_r^{\i}\bar{f}_{\th,k} (s)s^{1-\nu}\mK_k(s)\md s\,.
\]
Using \eqref{bessel1} and \eqref{bessel2}, we have
\begin{align}\label{FFFF0}
|\mathcal{L}_\i(r)| &\ls_\nu |k|^{-\f{1}{2}}r^{\f{\nu-1}{2}}e^{|k|r} \|\bar{f}_{\th,k}\|_{L^{\i}_{\tilde{\la}}}\int^{+\i}_r s^{1-\nu-{\tilde{\la}}}   |k|^{-\f{1}{2}}s^{\f{\nu-1}{2}}e^{-|k|s} \md s\nn\\
&\ls_{\nu,{\tilde{\la}}} |k|^{-2}r^{\f{\nu-1}{2}}e^{|k|r} \|\bar{f}_{\th,k}\|_{L^{\i}_{\tilde{\la}}} r^{\f{1-\nu}{2}-{\tilde{\la}}}e^{-|k|r}=|k|^{-2}\|\bar{f}_{\th,k}\|_{L^{\i}_{\tilde{\la}}}r^{-{\tilde{\la}}}\,.
\end{align}
Taking the first derivative of $\mathcal{L}_\i(r)$, we derive the following estimate similarly as we achieve \eqref{EEEE}:
\begin{align}\label{FFFF}
|\mathcal{L}^\prime_\i(r)|&\leq\lt|\mI^\prime_k(r)\int_r^{+\i} \bar{f}_{\th,k} (s)s^{1-\nu}\mK_k(s)\md s\rt|+\lt| \mI_k(r) \bar{f}_{\th,k} (r)r^{1-\nu}\mK_k(r)\rt| \nn\\[1mm]
                    &\ls_{\nu,{\tilde{\la}}} |k|^{-1}\|\bar{f}_{\th,k}\|_{L^{\i}_{\tilde{\la}}}r^{-{\tilde{\la}}}+ |k|^{-\f{1}{2}}r^{\f{\nu-1}{2}}e^{|k|r} \|\bar{f}_{\th,k}\|_{L^{\i}_{\tilde{\la}}} r^{-{\tilde{\la}}}r^{1-\nu}|k|^{-\f{1}{2}}r^{\f{\nu-1}{2}}e^{-|k|r}\nn\\[1mm]
                     &\ls_{\nu,{\tilde{\la}}} |k|^{-1}\|\bar{f}_{\th,k}\|_{L^{\i}_{\tilde{\la}}}r^{-{\tilde{\la}}}\,.
\end{align}
Then we find
\be\label{linearvth4}
k^2|\mathcal{L}_\i(r)|+|k|\cd|\mathcal{L}^\prime_\i(r)|\ls_{\nu,{\tilde{\la}}} \|\bar{f}_{\th,k}\|_{L^{\i}_{\tilde{\la}}}r^{-{\tilde{\la}}}\,.
\ee
by \eqref{FFFF0} and \eqref{FFFF}. Combining estimates in \eqref{linearvth2}, and \eqref{linearvth3} and \eqref{linearvth4}, we conclude that
\be\label{linearvth5}
k^2|v_{\th,k}(r)|+|k|\cd|v^\prime_{\th,k}(r)|\ls_{\nu,{\tilde{\la}}} \|\bar{f}_{\th,k}\|_{L^{\i}_{\tilde{\la}}}r^{-{\tilde{\la}}}+ e^{-r}k^2|g_{\th,k}|\,.
\ee
From \eqref{elinear}$_1$, we have
\be\label{linearvth6}
|v^{\prime\prime}_{\th,k}(r)|\ls_\nu r^{-1}|v^{\prime}_{\th,k}(r)|+ k^2 |v_{\th,k}(r)|+|\bar{f}_{\th,k}|\,.
\ee
Thus one concludes
\be\label{linearvth7}
k^2|v_{\th,k}(r)|+|k|\cd|v^\prime_{\th,k}(r)|+|v^{\prime\prime}_{\th,k}(r)|\ls_{\nu,{\tilde{\la}}} \|\bar{f}_{\th,k}\|_{L^{\i}_{\tilde{\la}}}r^{-{\tilde{\la}}}+ e^{-r}k^2|g_{\th,k}|\,.
\ee
By combining \eqref{linearvth5} and \eqref{linearvth6}. This completes the proof of the proposition.

$\hfill\square$

\begin{remark}
Here we only need $\t{\la}>1$ to guarantee the validity of the asymptotic condition
\[
v_{\th,k}(r)=o(r^{-1})\,\q\text{as}\q r\to\i\,.
\]
However, this is not enough for us to construct the nonlinear solution. See Section \ref{sec4} below. That is why we impose the condition $\la>3/2$ in the main theorem.
\end{remark}

$\hfill\square$

}

\subsection{The non-zero mode of $\bl{(v_r,v_z)}$  and their estimates}
Since $v_{\th}$ is solved, now we are ready for the remaining $v_r$ and $v_z$, which satisfies:
 \be\label{elinearrz}
\left\{\begin{aligned}
&-\left(\f{\md^2}{\md r^2}+\f{1-{\nu}}{r}\f{\md}{\md r}-\f{1-{\nu}}{r^2}-k^2\right)v_{r,k}+\f{\md}{\md r}\pi_k=\bar{f}_{r,k},\q&\text{for}\q r\in]1,\i[\,,\\[1mm]
&-\left(\f{\md^2}{\md r^2}+\f{1-{\nu}}{r}\f{\md}{\md r}-k^2\right)v_{z,k}+ik\pi_k=\bar{f}_{z,k}\,,\q&\text{for}\q r\in]1,\i[\,,\\[1mm]
&ikv_{z,k}+\f{\md}{\md r}v_{r,k}+\f{v_{r,k}}{r}=0\,,\q&\text{for}\q r\in]1,\i[\,,\\[1mm]
&v_{r,k}(1)=g_{r,k}\,\q v_{z,k}(1)=g_{z,k}\,,\q v_{r,k}(+\i)=v_{z,k}(+\i)=0.&
\end{aligned}\right.
\ee

Now we give the main result of this subsection:
\begin{proposition}\label{Prop3}
For $k\in\mathbb{Z}-\{0\}$, given $\bar{f}_{r,k},\,\bar{f}_{z,k}\in L^\i_{\la}$ with $\la>3/2$. The boundary value problem \eqref{elinearrz} has a unique solution
\[
(v_{r,k}(r)\,,\,v_z(r))=o(r^{-1})\,\q\text{as}\q r\to\i\,.
\]
And it satisfies the estimate
\be\label{Mainvthk}
\|(v_{r,k}'',v_{z,k}'')\|_{L^\i_{\t{\la}}}+|k|\cd\|(v_{r,k}',v_{z,k}')\|_{L^\i_{\t{\la}}}+k^2\|(v_{r,k},v_{z,k})\|_{L^\i_{\t{\la}}}\les_{\nu,\t{\la}}\|(\bar{f}_{r,k},\bar{f}_{z,k})\|_{L^\i_{\t{\la}}}+k^2\left(|g_{r,k}|+|g_{z,k}|\right)\,
\ee
for any $\t{\la}\in\,]\f{3}{2},\la]$ .
\end{proposition}

In each Fourier mode, to overcome the difficulty caused by the pressure, we introduce the stream function and vorticity to the linearized system \eqref{PoNSvL00}$_{1,3}$. Since the related linear velocity field $\bl{\mathfrak{b}}:=v_r\bl{e_r}+v_z\bl{e_z}$ is divergence-free, there exists a periodic stream function $\phi$ in the $z$ variable such that
\[
v_r=-\p_z\phi,\q\text{and}\q v_z=\f{1}{r}\p_r(r\phi)\,,
\]
and define the vorticity $w$ by
\[
w:=\p_zv_r-\p_rv_z\,,
\]
which satisfies
\[
-\left(\Dl-\f{1}{r^2}\right)\phi=w.
\]
Using \eqref{PoNSvL00}$_{1,3}$, we derive $w$ satisfies
\be\label{Evor}
-\left(\p_r^2+\f{1-{\nu}}{r}\p_r-\f{1-{\nu}}{r^2}+\p_z^2\right)w=F=\p_z\bar{f}_r-\p_r\bar{f}_z.
\ee
After the vorticity $w$ is achieved, we can solve the stream function $\phi$ which satisfies
\be\label{Ephi}
-\left(\p_r^2+\f{1}{r}\p_r+\p_z^2-\f{1}{r^2}\right)\phi=w\,,
\ee
and recover $v_r$ and $v_z$ by
\be\label{vphi}
v_r=-\p_z\phi,\q\text{and}\q v_z=\f{1}{r}\p_r(r\phi)\,.
\ee
Writing \eqref{Evor}--\eqref{Ephi} in each Fourier mode, we deduce the following system:
\be\label{streamvor}
\lt\{
\begin{aligned}
&-\left(\f{\md^2}{\md r^2}+\f{1}{r}\f{\md}{\md r}-k^2-\f{1}{r^2}\right)\phi_k(r)=w_k(r)\,,&\q\text{in}\q r\in]1,\i[\,,\\[1mm]
&-\left(\f{\md^2}{\md r^2}+\f{1-{\nu}}{r}\f{\md}{\md r}-\f{1-{\nu}}{r^2}-k^2\right)w_{k}(r)=F_{k}:=ik\bar{f}_{r,k}(r)-\bar{f}_{z,k}'(r)\,,&\q\text{in}\q r\in]1,\i[\,,\\[1mm]
&\phi_k(1)=\bar{\phi}_k,\q w_k(1)=\bar{w}_k,\q\phi_k(+\i)=w_k(+\i)=0\,.&
\end{aligned}
\rt.
\ee
System \eqref{streamvor} can also be deduce from \eqref{elinearrz} by eliminating $\pi_k$. Now we only consider the case of $k\neq0$ in this subsection, while the rest will be discussed later. The homogeneous vorticity equation in \eqref{streamvor} reads:
\[
-\left(\frac{\md^2}{\md r^2}+\frac{1-\nu}{r} \frac{\md}{\md r}-\frac{1-\nu}{r^2}-k^2\right) \tilde{w}_k(r)=0\,.
\]
Similarly as we handled the equation \eqref{BE1} before, we rewrite
\[
\tilde{w}_k(r)=r^{\f{\nu}{2}}h(|k|r)\,,
\]
and direct calculation shows $h$ satisfies the following modified Bessel equation
\[
-h''-\f{1}{r}h'+\Big(1+\f{\left(1-\f{\nu}{2}\right)^2}{r^2}\Big)h=0\,.
\]
As we considered in the previous subsection, it has two linearly independent solutions $I_{\left|1-\f{\nu}{2}\right|}(r)$ and $K_{\left|1-\f{\nu}{2}\right|}(r)$, which are exponential growing and exponential decaying at $r\to\i$, respectively. Akin to \eqref{FUnc}, we denote
\[
\mathfrak{K}_k(r):=r^{\f{\nu}{2}}K_{\left|1-\f{\nu}{2}\right|}(|k|r)\,,\q \mathfrak{I}_k(r):=r^{\f{\nu}{2}}I_{\left|1-\f{\nu}{2}\right|}(|k|r)\,.
\]
Similarly as \eqref{subn0}, the solution of \eqref{streamvor}$_2$ that decays faster enough at spacial infinity can be represented by  for $k\in\bZ-\{0\}$,
\be\label{subn0w}
\begin{split}
w_k(r)=&\bar{w}_k\mathfrak{K}_k(r)+\mathfrak{K}_k(r)\int_1^rF_k(s)s^{1-\nu}\mathfrak{I}_k(s)\md s+\mathfrak{I}_k(r)\int_r^{\i}F_k(s)s^{1-\nu}\mathfrak{K}_k(s)\md s.
\end{split}
\ee

After the $w_k$ is achieved, we can solve the $k$ mode of the stream function as  for $k\in\bZ-\{0\}$,
\be\label{Sphi}
\phi_k(r)={\bar{\phi}_k}K_1(|k|r)+K_1(|k|r)\int_1^{r} I_1(|k|s)sw_k(s)\mathrm{d} s+I_1(|k|r)\int_r^\i K_1(|k|s)sw_k(s)\mathrm{d} s.
\ee

Notice that the pair of constants $(\bar{w}_k\,,\,\bar{\phi}_k)$ is not known at the moment. We now intend to obtain it by applying the boundary condition of ${v_{r,k}}$ and $v_{z,k}$. Recall \eqref{vphi} and the boundary condition \eqref{elinear}$_4$, we derives
\be\label{Bf}
g_{r,k}=-ik\phi_k(1),\q\text{and}\q g_{z,k}=\phi_k'(1)+\phi_k(1)\,.
\ee
Substituting \eqref{Bf} in \eqref{Sphi}, we derive
\[
\left\{
\begin{split}
&-\f{g_{r,k}}{ik}=\bar{\phi}_kK_1(|k|)+I_1(|k|)\int_1^\i K_1(|k|s)sw_k(s)\mathrm{d} s\,,\\
&g_{z,k}=\bar{\phi}_k\left\{K_1(|k|)+|k|K_1'(|k|)\right\}+\left(I_1(|k|)+|k|I_1'(|k|)\right)\int_1^\i K_1(|k|s)sw_k(s)\mathrm{d} s\,,
\end{split}
\right.\q\text{for}\q k\in\mathbb{Z}-\{0\}\,.
\]
Eliminating the terms of $w_k$ and $\bar{\phi}_k$ from the above equations, we derive that for any $k \in \mathbb{Z}-\{0\}$
\be\label{Sphik}
\left\{
\begin{split}
&\bar{\phi}_k=-g_{z,k}I_1(|k|)-\f{g_{r,k}}{ik}\left(I_1(|k|)+|k|I_1'(|k|)\right)\,,\\[1mm]
&\int_1^\i K_1(|k|s)sw_k(s)\md s=\f{g_{r,k}}{ik}I_1^{-1}(|k|)\left(I_1(|k|)K_1(|k|)+|k|I_1'(|k|)K_1(|k|)-1\right)+g_{z,k}K_1(|k|)\,.
\end{split}
\right.
\ee
Here we have applied the following identity for modified Bessel functions\footnote{See page 375 in \cite{AS1992}.}:
\[
K_1'(x)I_1(x)-K_1(x)I_1'(x)=-x^{-1},\q\text{for}\q x>0\,.
\]
Inserting \eqref{subn0w} in \eqref{Sphik}$_2$, one deduces
\be\label{swk}
\begin{aligned}
&\bar{w}_k \int_1^{\infty} K_1(|k| s)s\mathfrak{R}_k(s) \md s+\int_1^{\infty} K_1(|k|s)s h_{k, F}(s) \md s \\
= & \frac{g_{r, k}}{i k} I_1^{-1}(|k|)\left(I_1(|k|) K_1(|k|)+|k| I_1^{\prime}(|k|) K_1(|k|)-1\right)+g_{z, k} K_1(|k|)\,,
\end{aligned}
\ee
with
\be\label{HKF}
h_{k,F}=\mathfrak{K}_k(r)\int_1^rF_k(s)s^{1-\nu}\mathfrak{I}_k(s)\md s+\mathfrak{I}_k(r)\int_r^{\i}F_k(s)s^{1-\nu}\mathfrak{K}_k(s)\md s\,.
\ee
For simplicity, we denote
\begin{align}
& A_k=\frac{1}{i k} I_1^{-1}(|k|)\left(I_1(|k|) K_1(|k|)+|k| I_1^{\prime}(|k|) K_1(|k|)-1\right)\,;\nn \\[1mm]
& B_k=K_1(|k|)\,; \nn\\[1mm]
& D_k=\int_1^{\infty} K_1(|k| s) s\mathfrak{R}_k(s)\md s\,;\nn\\[1mm]
& G_{k,F}=\int_1^{\infty} K_1(|k|s)s h_{k, F}(s) \md s\,.\label{abdg}
\end{align}
Then we infer from \eqref{swk} that
\be\label{Ewk}
\bar{w}_k=D_k^{-1}\left(A_kg_{r,k}+B_kg_{z,k}-G_{k,F}\right)\,.
\ee

{
Akin to Lemma \ref{lemKI}, one has the following estimates for $\mathfrak{K}_k$ and $\mathfrak{I}_k$. We list the lemma here without proof.

\begin{lemma}\label{lemfKfI}
The linearly independent fundamental solutions $\mathfrak{K}_k(r)$ and $\mathfrak{I}_k(r)$ have the following estimates

\be\label{bessel7}
\f{\md^n}{\md r^n}\fK_k(r)\les_\nu |k|^{-\f{1}{2}+n}r^{\f{\nu-1}{2}}e^{-|k|r},\q \f{\md^n}{\md r^n}\fI_k(r)\les_\nu |k|^{-\f{1}{2}+n}r^{\f{\nu-1}{2}}e^{|k|r}\,,
\ee
for $n=0,\,1,\,2$ .
\end{lemma}
$\hfill\square$

Now we are ready for deriving the upper bound of $\bar{\phi}_k$ and $\bar{w}_k$.
\begin{lemma}
For $\t{\la}\in]3/2,\la]$, the following estimates of $\bar{\phi}_k$ and $\bar{w}_k$ hold:
\be\label{vor7}
\begin{split}
|\bar{\phi}_{k}|&\ls_\nu |k|^{-1/2}\lt(|g_{r,k}|+|g_{z,k}|\rt)e^{|k|}\,;\\[1mm]
|\bar{w}_k|&\ls_\nu |k|^{\f{3}{2}}e^{|k|}\lt(|g_{r,k}|+|g_{z,k}|\rt)+ |k|^{-\f{1}{2}}e^{|k|} \|(\bar{f}_{r,k},\bar{f}_{z,r})\|_{L^{\i}_{\tilde{\la}}}\,.
\end{split}
\ee
\end{lemma}
\begin{proof}
Recall \eqref{Sphik}$_1$,  the estimate \eqref{vor7}$_1$ is a direct corollary of \eqref{bessel}. Consider \eqref{HKF}. By a similar approach as we derive \eqref{linearvth3} and \eqref{linearvth4}, we have
\[
|h_{k,F}(s)|\ls |k|^{-1}\|(\bar{f}_{r,k},\bar{f}_{z,r})\|_{L^{\i}_{\tilde{\la}}}s^{-{\tilde{\la}}}\,.
\]
Then similarly we derive
\be\label{vor008}
|G_{k,F}|\ls |k|^{-\f{5}{2}}e^{-|k|}\|(\bar{f}_{r,k},\bar{f}_{z,r})\|_{L^{\i}_{\tilde{\la}}}.
\ee
Meanwhile, from the representations in \eqref{abdg} and using estimates in Lemma \ref{lemKI} and Lemma \ref{lemfKfI}, we have
\be\label{vor8}
|A_k|\ls |k|^{-\f{1}{2}}e^{-|k|},\q |B_k|\thickapprox |k|^{-\f{1}{2}}e^{-|k|},\q |D_k|\thickapprox_\nu |k|^{-2}e^{-2|k|}\,,
\ee
Using \eqref{vor008} and \eqref{vor8}, we infer from \eqref{Ewk} that
\[
|\bar{w}_k|\les_\nu k^2e^{2|k|}\left\{|k|^{-\f{1}{2}}e^{-|k|}\left(|g_{r,k}|+|g_{z,k}|\right)+|k|^{-\f{5}{2}}e^{-|k|} \|(\bar{f}_{r,k},\bar{f}_{z,r})\|_{L^{\i}_{\tilde{\la}}}\right\}\,.
\]
This proves \eqref{vor7}$_2$.
\end{proof}

\leftline{\emph{Proof of Proposition \ref{Prop3}.}}
Now we are on the way for the estimate of $w_k$.  Recall that
\[
\begin{split}
w_k(r)=&\bar{w}_k\mathfrak{K}_k(r)+\mathfrak{K}_k(r)\int_1^r\lt(ik\bar{f}_{r,k}(s)-\bar{f}'_{z,k}(s)\rt)s^{1-\nu}\mathfrak{I}_k(s)\md s\nn\\
      &+\mathfrak{I}_k(r)\int_r^{\i}\lt(ik\bar{f}_{r,k}(s)-\bar{f}'_{z,k}(s)\rt)s^{1-\nu}\mathfrak{K}_k(s)\md s\,.
\end{split}
\]
Direct integration by parts imply that
\begin{align}\label{vor1}
w_k(r)=&\bar{w}_k\mathfrak{K}_k(r)+\mathfrak{K}_k(r)\int_1^rik\bar{f}_{r,k}(s)s^{1-\nu}\mathfrak{I}_k(s)\md s+\mathfrak{I}_k(r)\int_r^{\i}ik\bar{f}_{r,k}(s)s^{1-\nu}\mathfrak{K}_k(s)\md s\nn\\
       &-\mathfrak{K}_k(r) s^{1-\nu}\mathfrak{I}_k(s)\bar{f}_{z,k}(s)\big|^r_1-\mathfrak{I}_k(r) s^{1-\nu}\mathfrak{K}_k(s)\bar{f}_{z,k}(s)\big|^{+\i}_r\nn\\
       &+\mathfrak{K}_k(r)\int_1^r \lt(s^{1-\nu}\mathfrak{I}_k(s)\rt)^\prime \bar{f}_{z,k}(s) \md s+\mathfrak{I}_k(r)\int_1^r \lt(s^{1-\nu}\mathfrak{K}_k(s)\rt)^\prime \bar{f}_{z,k}(s) \md s\nn\\
      =&\underbrace{\lt(\bar{w}_k+\mathfrak{I}_k(1)\bar{f}_{z,k}(1)\rt)\mathfrak{K}_k(r)}_{\mM_1}\nn\\
       &+\underbrace{\mathfrak{K}_k(r)\int_1^rik\bar{f}_{r,k}(s)s^{1-\nu}\mathfrak{I}_k(s)\md s}_{\mM_2}+\underbrace{\mathfrak{I}_k(r)\int_r^{\i}ik\bar{f}_{r,k}(s)s^{1-\nu}\mathfrak{K}_k(s)\md s}_{\mM_3}\nn\\
       &+\underbrace{\mathfrak{K}_k(r)\int_1^r \lt(s^{1-\nu}\mathfrak{I}_k(s)\rt)^\prime \bar{f}_{z,k}(s) \md s}_{\mM_4}+\underbrace{\mathfrak{I}_k(r)\int_1^r \lt(s^{1-\nu}\mathfrak{K}_k(s)\rt)^\prime \bar{f}_{z,k}(s) \md s}_{\mM_5}.
\end{align}
Using \eqref{bessel7}, we see that
\be\label{linearvor2}
|\mM^\prime_1(r)|+|k|\cd|\mM_1(r)|\ls_\nu \lt|\bar{w}_k+\mathfrak{I}_k(1)f_{z,k}(1)\rt| r^{\f{\nu-1}{2}} |k|^{\f{1}{2}}e^{-|k|r}.
\ee
By a similar approach as we derive \eqref{linearvth3} and \eqref{linearvth4}, we have
\be\label{linearvor3}
\lt|\mathcal{M}^\prime_j(r)\rt|+|k|\cd|\mathcal{M}_j(r)|\ls_{\nu,{\tilde{\la}}}\left\{
\begin{aligned}
&\|f_{r,k}\|_{L^\i_{\tilde{\la}}}r^{-{\tilde{\la}}}+|\mu|\cd\|v_{\th,k}\|_{L^{\i}_{\tilde{\la}}}r^{-{\tilde{\la}}-2}\,,\q&\text{for}\q j=2,3\,;\\
&\|f_{z,k}\|_{L^{\i}_{\tilde{\la}}}r^{-{\tilde{\la}}}\,,\q&\text{for}\q j=4,5\,.\\
\end{aligned}
\right.
\ee
Inserting \eqref{linearvor2} and \eqref{linearvor3} into \eqref{vor1}, we obtain that
\begin{align}
|w^\prime_k(r)|+|k|\cd|w_k(r)|\les_{\nu,{\tilde{\la}}}\|(\bar{f}_{r,k},\bar{f}_{z,r})\|_{L^{\i}_{\tilde{\la}}}r^{-{\tilde{\la}}}+ \lt|\bar{w}_k+\mathfrak{I}_k(1)f_{z,k}(1)\rt| r^{\f{\nu-1}{2}} |k|^{\f{1}{2}}e^{-|k|r}. \label{vor2}
\end{align}

For the estimates of $\phi_{k}$, from the representation \eqref{Sphi} and follow same approach as we achieve \eqref{linearvth7}, we can obtain that

\be\label{vor3}
|k|^3|\phi_{k}(r)|+|k|^2|\phi^\prime_{k}(r)|+|k|\cd|\phi^{\prime\prime}_{k}(r)|\ls_{\tilde{\la}} |k|\cd\|w_{k}\|_{L^{\i}_{\tilde{\la}}}r^{-{\tilde{\la}}}+ |\bar{\phi}_k| r^{\f{\nu-1}{2}} |k|^{\f{5}{2}}e^{-|k|r}.
\ee

Now we are ready for estimates of $v_{r,k}$ and $v_{z,k}$. From the relation between $\phi$ and $v_r,v_z$, we have
\be\label{vor4}
\left\{
\begin{aligned}
&v_{r,k}=-ik\phi_k\,,\\
&v^\prime_{r,k}=-ik\phi^\prime_k\,,\\
&v^{\prime\prime}_{r,k}=-ik\phi^{\prime\prime}_k\,,
\end{aligned}
\right.\q\q\q\q\left\{
\begin{aligned}
&v_{z,k}=\phi^\prime_k+ r^{-1}\phi_k\,,\\
&v^\prime_{z,k}=\phi^{\prime\prime}_k+ r^{-1}\phi^\prime_k-r^{-2}\phi_k\,,\\
&v^{\prime\prime}_{z,k}=-w^\prime_k+ikv^{\prime}_{r,k}=-w^\prime_k+k^2\phi^\prime_k\,.
\end{aligned}
\right.
\ee
Combining \eqref{vor3} and \eqref{vor4}, we see that
\be\label{vor5}
\begin{split}
&k^2|(v_{r,k},v_{z,k})(r)|+|k|\cd|(v^\prime_{r,k},v^\prime_{z,k})(r)|+|(v^{\prime\prime}_{r,k},v^{\prime\prime}_{z,k})(r)|\\
\ls_{\tilde{\la}}& |k|^3|\phi_{k}(r)|+|k|^2|\phi^\prime_{k}(r)|+|k|\cd|\phi^{\prime\prime}_{k}(r)|+|w^\prime_k(r)|\\
\ls_{\tilde{\la}}& |k|\cd\|w_{k}\|_{L^{\i}_{\tilde{\la}}}r^{-{\tilde{\la}}}+ |\bar{\phi}_k| r^{\f{\nu-1}{2}} |k|^{\f{5}{2}}e^{-|k|r}+|w^\prime_k(r)|\,.
\end{split}
\ee
Inserting estimates in \eqref{vor2} into \eqref{vor5}, we can obtain that
\begin{align}
&k^2|(v_{r,k},v_{z,k})(r)|+|k|\cd|(v^\prime_{r,k},v^\prime_{z,k})(r)|+|(v^{\prime\prime}_{r,k},v^{\prime\prime}_{z,k})(r)|\nn\\
\ls_{\nu,{\tilde{\la}}}\,&  |\bar{\phi}_k| r^{\f{\nu-1}{2}} |k|^{\f{5}{2}}e^{-|k|r}+\lt|\bar{w}_k+\mathfrak{I}_k(1)f_{z,k}(1)\rt| r^{\f{\nu-1}{2}} |k|^{\f{1}{2}}e^{-|k|r}+\|(\bar{f}_{r,k},\bar{f}_{z,r})\|_{L^{\i}_{\tilde{\la}}}r^{-{\tilde{\la}}}\nn\\
  \ls_{\nu,{\tilde{\la}}}\,&  |\bar{\phi}_k| r^{\f{\nu-1}{2}} |k|^{\f{5}{2}}e^{-|k|r}+\lt|\bar{w}_k\rt| r^{\f{\nu-1}{2}} |k|^{\f{1}{2}}e^{-|k|r}+\|(\bar{f}_{r,k},\bar{f}_{z,r})\|_{L^{\i}_{\tilde{\la}}}r^{-{\tilde{\la}}}. \label{vor6}
\end{align}
Inserting \eqref{vor7} into \eqref{vor6}, we finally conclude that
\begin{align}
&k^2|(v_{r,k},v_{z,k})(r)|+|k|\cd|(v^\prime_{r,k},v^\prime_{z,k})(r)|+|(v^{\prime\prime}_{r,k},v^{\prime\prime}_{z,k})(r)|\nn\\
  \ls&  k^{2}\lt(|g_{r,k}|+|g_{z,k}|\rt)e^{-r}+\|(\bar{f}_{r,k},\bar{f}_{z,r})\|_{L^{\i}_{\tilde{\la}}}r^{-{\tilde{\la}}}, \nn
\end{align}
which indicates \eqref{Mainvthk}. This complete the proof of Proposition \ref{Prop3}.
$\hfill\square$

}

\section{The nonlinear estimate and proof of the main result}\label{sec4}

\subsection{Existence}

In this section, we solve the nonlinear problem \eqref{ReduNS} by applying the fixed point theorem. Let
\[
\bar{\bl{v}}=\bar{v}_r(r,z)\bl{e_r}+\left(\bar{v}_\th(r,z)+\f{\bar{\sigma}}{r}{\bl{1}}_{-2\leq\nu<0}\right)\bl{e_\th}+\bar{v}_z(r,z)\bl{e_z}\in {\mathcal{B}}_\tau\,,
\]
where
\[
\tau=\min\left\{\bar{\la_\th}-3\,,\,\bar{\la_z}-2\,,\,\la-\f{3}{2}\right\}\,.
\]
Here $\bar{\sigma}\in\mathbb{R}$, $\bar{v}_\th=o(r^{-1})\,,$ as $r\to\i\,$. The definition of $\bar{\la_\th}$ and $\bar{\la_z}$ are given in \eqref{lath} and \eqref{laz}, respectively.

We consider the following linear system
\begin{small}
\be\label{PoNSvN}
\left\{\begin{split}
&-\left(\p_r^2+\f{1-\nu}{r}\p_r-\f{1-\nu}{r^2}+\p_z^2\right)v_r+\p_r\pi={\f{2\mu}{r^2}{v}_\th}+\bar{f}_r\,,\\
& -\left(\p_r^2+\f{1-\nu}{r}\p_r-\f{1+\nu}{r^2}+\p_z^2\right)v_\th=\bar{f}_\th\,,\\
&-\left(\p_r^2+\f{1-\nu}{r}\p_r+\p_z^2\right)v_z+\p_z\pi=\bar{f}_z\,,\\
&\p_z(r v_z)+\p_r(rv_r)=0\,,\\
&v_r\big|_{r=1}={g}_r,\q v_\th\big|_{r=1}= g_\th,\q v_z\big|_{r=1}= g_z, \q \bl{v}\big|_{r\rightarrow+\i}=0\,,
\end{split}\right.
\ee
\end{small}
where

\be\label{RHS}
\left\{
\begin{split}
&\bar{f}_r:=-\left(\bar{v}_r\p_r+{\bar{v}_z}\p_z\right)\bar{v}_r+\f{\bar{v}_\th^2}{r}+\f{2\bar{\sigma}\bar{v}_\th}{r^2}{\bl{1}}_{-2\leq\nu<0}+f_r\,;\\[1mm]
&\bar{f}_\th:=-\left(\bar{v}_r\p_r+{\bar{v}_z}\p_z\right)\bar{v}_\th-\f{\bar{v}_r\bar{v}_\th}{r}+f_\th\,;\\[1mm]
&\bar{f}_z:=-\left(\bar{v}_r\p_r+{\bar{v}_z}\p_z\right)\bar{v}_z+f_z\,.
\end{split}\right.
\ee
We mention here the extra term $(\f{\bar{\sigma}}{r})^2{\bl{1}}_{-2\leq\nu<0}$ is absorbed in $f_{r,0}$ which has no restriction.

Recall linear estimates the previous section, the system \eqref{ReduNS} has a unique solution $\bl{v}\in\mathcal{B}_\tau$. There we choose
\[
\t{\la_\th}=3+\tau\,,\q\t{\la}=\f{3}{2}+\tau\,,\q\t{\la_z}=2+\tau\,.
\]
Proposition \ref{Prop1} and Proposition \ref{Prop2} indicate $v_\th$ satisfies:
\be\label{nonlinear1}
\begin{split}
&|\sigma|{\bf1}_{-2<\nu<0}+\sum_{\ell\in\{0,1,2\}}\|\mathrm{v}^{(2-\ell)}(r)\|_{L^\i_{3+\tau-\ell}}+\sum_{k\neq0,\ell\in\{0,1,2\}}|k|^{2-\ell}\|v^{(\ell)}_{\th,k}(r)\|_{L^\i_{3/2+\tau}}\\
\ls_{\nu,{\la_\th},\la_z,{\la}}& \,\|\bar{f}_{\th,0}\|_{L^\i_{3+\tau}}+\sum_{k\neq 0}\|\bar{f}_{\th,k}\|_{L^\i_{3/2+\tau}}+ \sum_{k\in\bZ}(1+k^2)|g_{\th,k}|\,.
\end{split}
\ee
Moreover, by Proposition \ref{Prop3} and Proposition \ref{Prop4}, $(v_r,v_z)$ satisfies:
\be\label{nonlinear2}
\begin{split}
&\|v^{(2-\ell)}_{z,0}(r)\|_{L^\i_{2+\tau-\ell}}+\sum_{k\neq0\atop j\in\{r,z\},\ell\in\{0,1,2\}}|k|^{2-\ell}\|v^{(\ell)}_{j,k}(r)\|_{L^\i_{3/2+\tau}}\\
\ls_{\nu,{\la_\th},\la_z,{\la}} & \, \|\bar{f}_{z,0}\|_{L^\i_{2+\tau}}+\sum_{k\neq 0}\|(\bar{f}_{r,k},\bar{f}_{z,k})\|_{L^\i_{3/2+\tau}}+|\mu|\sum_{k\neq0}\|v_{\th,k}\|_{L^{\i}_{3/2+\tau}}+ \sum_{k\in\bZ}(1+k^2)(|g_{r,k}|+|g_{z,k}| )\,.
\end{split}
\ee
For further estimate of $\bl{\bar{f}}$, we need the following lemma
\begin{lemma}
The following lemma of $\bar{\bl{f}}$ hold:
\begin{align}
 &\|\bar{f}_{\th,0}\|_{L^\i_{3+\tau}}+\sum_{k\neq 0}\|\bar{f}_{\th,k}\|_{L^\i_{3/2+\tau}}\ls  \|\bl{\bar{v}}\|^2_{\mathcal{B}_{\tau}}+\|\bl{f}\|_{\mathcal{E}_{\la_\th,\la_z,\la}}\,, \label{nonlinear3}\\
 &\|\bar{f}_{z,0}\|_{L^\i_{2+\tau}}+\sum_{k\neq 0}\|(\bar{f}_{r,k},\bar{f}_{z,k})\|_{L^\i_{3/2+\tau}}\ls\|\bl{\bar{v}}\|^2_{\mathcal{B}_{\tau}}+\|\bl{f}\|_{\mathcal{E}_{\la_\th,\la_z,\la}}\,. \label{nonlinear4}
\end{align}
\end{lemma}
\begin{proof}
Recall \eqref{RHS}$_2$. Using integration by parts with $z-$variable, and noticing the divergence-free property of $\bl{v}$, one derives:
\begin{align}
&\bar{f}_{\th,0}=-\f{1}{2\pi}\int_0^{2\pi}\left((\bar{v}_r\p_r\bar{v}_\th+\bar{v}_z\p_z\bar{v}_\th)+\f{\bar{v}_r\bar{v}_\th}{r}\right)\md z+f_{\th,0}\nn\\
=&-\sum\limits_{\ell\neq0} \lt(\bar{v}_{r,\ell} \p_r\bar{v}_{\th,-\ell}+\bar{v}_{z,\ell}i\ell\bar{v}_{\th,-\ell} +\f{\bar{v}_{r,\ell}}{r}\bar{v}_{\th,-\ell}\rt)+f_{\th,0}\,.\nn
\end{align}
This indicate that
\begin{align}
\|\bar{f}_{\th,0}\|_{L^\i_{3+\tau}}\ls\Big(\sum_{\ell\neq 0} |\ell|\cd\|\bar{v}_{r,\ell}, \bar{v}_{z,\ell} \|_{L^\i_{3/2+\tau}}\Big)\Big(\sum_{\ell\neq0}\|\bar{v}_{\th,\ell}, \p_r\bar{v}_{\th,\ell} \|_{L^\i_{3/2+\tau}}\Big)+\|{f}_{\th,0}\|_{L^\i_{3+\tau}}\,. \label{app1}
\end{align}
While for $k\neq 0$,
\begin{align}
\bar{f}_{\th,k}=-\sum_{\ell\in\mathbb{Z}} \lt(\bar{v}_{r,k-\ell} \p_r \bar{v}_{\th,\ell}+\bar{v}_{z,k-\ell}i\ell \bar{v}_{\th,\ell}+\f{\bar{v}_{r,k-\ell} \bar{v}_{\th,\ell}}{r}\rt)+{f}_{\th,k}\,.\nn
\end{align}
Using the Young inequality for convolution, we have
\begin{align}
\sum_{k\neq0}\|\bar{f}_{\th,k}\|_{L^\i_{3/2+\tau}}\ls \Big(\sum_{k\neq0}\| (\bar{v}_{r,k},kv_{\th,k})\|_{L^\i_{3/2+\tau}}\Big)\Big(\sum_{k\in\bZ}\| (\p_r\bar{v}_{\th,k},v_{z,k},\f{\bar{v}_{\th,k}}{r})\|_{L^\i}\Big)+ \sum_{k\neq0}\|{f}_{\th,k}\|_{L^\i_{3/2+\tau}}\,.\label{app2}
\end{align}
Thus \eqref{app1} and \eqref{app2} indicate \eqref{nonlinear3}. Meanwhile
\begin{align}
\bar{f}_{z,0}=-\sum_{\ell\neq0} \lt(\bar{v}_{r,-\ell} \p_r \bar{v}_{z,\ell}+\bar{v}_{z,-\ell} i\ell \bar{v}_{z,\ell}\rt)+{f}_{z,0}\,.\nn
\end{align}
Then we have
\begin{align}
\|\bar{f}_{z,0}\|_{L^\i_{2+\tau}}\ls \Big(\sum_{\ell\neq0}\| (\bar{v}_{r,\ell},kv_{z,\ell})\|_{L^\i_{3/2+\tau}}\Big)\Big(\sum_{\ell\in\mathbb{Z}}\| (\p_r\bar{v}_{z,\ell},v_{z,\ell})\|_{L^\i_{1/2}}\Big)+\|{f}_{z,0}\|_{L^\i_{2+\tau}}. \label{app3}
\end{align}
While for $k\neq0$,
\begin{align}
\bar{f}_{z,k}=-\sum_{\ell\in\mathbb{Z}} \lt(\bar{v}_{r,k-\ell} \p_r \bar{v}_{z,\ell}+\bar{v}_{z,k-\ell} \ell \bar{v}_{z,\ell}\rt)+{f}_{z,k}\,.\nn
\end{align}
Similarly as \eqref{app2}, we have
\begin{align}
\sum_{k\neq0}\|\bar{f}_{z,k}\|_{L^\i_{3/2+\tau}}\ls\Big(\sum_{k\neq0}\| (\bar{v}_{r,k},k \bar{v}_{z,k})\|_{L^\i_{3/2+\tau}}\Big)\Big(\sum_{k\in\bZ}\| (\p_r\bar{v}_{z,k},\bar{v}_{z,k})\|_{L^\i}\Big) + \sum_{k\neq0}\|{f}_{z,k}\|_{L^\i_{3/2+\tau}}. \label{app5}
\end{align}
The term $\bar{f}_{r,0}$ can be absorbed in the zero mode of pressure, one only consider $\bar{f}_{r,k}$ for $k\neq 0$. Clearly
\begin{align}
\bar{f}_{r,k}=-\sum_{\ell\in\mathbb{Z}} \lt(\bar{v}_{r,k-\ell} \p_r \bar{v}_{r,\ell}+\bar{v}_{z,k-\ell} i\ell \bar{v}_{r,\ell}-\f{\bar{v}_{\th,k-\ell} \bar{v}_{\th,\ell}}{r}\rt)+\f{2\bar{\sigma}\bar{v}_{\th,k}}{r^2}{\bl{1}}_{-2\leq\nu<0}+{f}_{r,k}\,.\nn
\end{align}
Then we have
\begin{align}
\sum_{k\neq0}\|\bar{f}_{r,k}\|_{L^\i_{3/2+\tau}}\ls& \Big(\sum_{k\neq0}\| (\bar{v}_{r,k},k \bar{v}_{r,k}, \bar{v}_{\th,k})\|_{L^\i_{3/2+\tau}}\Big)\Big(\sum_{k\in\bZ}\| (\p_r\bar{v}_{r,k},\bar{v}_{z,k},\bar{v}_{\th,k})\|_{L^\i}\Big)\nn\\
&+{\bl{1}}_{-2\leq\nu<0}|\bar{\sigma}|\sum_{k\neq0}\|\bar{v}_{\th,k}\|_{L^\i_{3/2+\tau}} + \sum_{k\neq0}\|{f}_{r,k}\|_{L^\i_{3/2+\tau}}\,. \label{app4}
\end{align}
Thus \eqref{app3}, \eqref{app4} and \eqref{app5} indicate \eqref{nonlinear4}. This finishes the proof of the lemma.
\end{proof}

By calculating
\[
\eqref{nonlinear1}+\dl\times\eqref{nonlinear2}
\]
for some small $0<\dl<1$, then using \eqref{nonlinear3} and \eqref{nonlinear4}, we obtain that

\be\label{ZZGJ1}
\begin{split}
|\sigma|{\bf1}_{-2<\nu<0}&+\sum_{\ell\in\{0,1,2\}}\|\mathrm{v}^{(2-\ell)}(r)\|_{L^\i_{3+\tau-\ell}}+\sum_{k\neq0,\ell\in\{0,1,2\}}|k|^{2-\ell}\|v^{(\ell)}_{\th,k}(r)\|_{L^\i_{3/2+\tau}}\\
&+\dl\,\Big(\|v^{(2-\ell)}_{z,0}(r)\|_{L^\i_{2+\tau-\ell}}+\sum_{k\neq0\atop j\in\{r,z\},\ell\in\{0,1,2\}}|k|^{2-\ell}\|v^{(\ell)}_{j,k}(r)\|_{L^\i_{3/2+\tau}}\Big)\\
\ls_{\nu,{\la_\th},\la_z,{\la}} \|\bl{\bar{v}}\|^2_{\mathcal{B}_{\tau}}&+\dl|\mu|\sum_{k\neq0}\|v_{\th,k}\|_{L^{\i}_{3/2+\tau}}+\|\bl{f}\|_{\mathcal{E}_{\la_\th,\la_z,\la}}+\|\bl{g}\|_{\mathcal{V}}\,.
\end{split}
\ee
Choosing suitably small $\dl$, we infer from \eqref{ZZGJ1} that there exists a constant $C>0$ which may depend on $\nu$, $\mu$, $\la_\th$, $\la_z$ and $\la$, such that
\be\label{nonlinear9}
\|\bl{{v}}\|_{\mathcal{B}_{\tau}}\leq C_{\nu,\mu,\la_\th,\la_z,\la} \left(\|\bl{\bar{v}}\|^2_{\mathcal{B}_{\tau}}+\|\bl{f}\|_{\mathcal{E}_{\la_\th,\la_z,\la}}+\|\bl{g}\|_{\mathcal{V}}\right)\,.
\ee

In the following, we omit the lower indexes of the constant $C$ for the convenience. The above process in this section indicates that:

\begin{framed}
\emph{Given any $\bar{\bl{v}}\in{\mathcal{B}_{\tau}}$ with $\|\bar{\bl{v}}\|_{{\mathcal{B}_{\tau}}}<2C\e$, and $(\bl{f},\bl{g})\in\mathcal{E}_{\la_\th,\la_z,\la}\times\mathcal{V}$ such that
\[
\|\bl{f}\|_{\mathcal{E}_{\la_\th,\la_z,\la}}+\|\bl{g}\|_{\mathcal{V}}<\e\,,
\]
where $\e<(4C)^{-2}$, there exists a unique solution $\bl{v}$ of \eqref{PoNSvN} that satisfies
\[
\|\bl{v}\|_{{\mathcal{B}_{\tau}}}\leq C\left(4C^2\e^2+\e\right)<2C\e\,.
\]}
\end{framed}

Therefore, the solution map
\[
\mathrm{T}\q:\q\bar{\bl{v}}\q\to\q\bl{v}
\]
arise from problem \eqref{PoNSvN} maps the ball
 \[
\mathbf{B}_{2C\e}:=\{\bl{v}\in{\mathcal{B}_{\tau}}\,:\,\|\bl{v}\|_{{\mathcal{B}_{\tau}}}<2C\e\}
 \]
 to itself.

 On the other hand, given any distinct
 \[
 \bar{\bl{v}}_1\,,\,\bar{\bl{v}}_2\in \mathbf{B}_{2C\e}\,,
 \]
the same estimate as obtain \eqref{nonlinear9} indicates that
\be\label{nlin10}
\begin{split}
\|\mathrm{T}\bar{\bl{v}}_1-\mathrm{T}\bar{\bl{v}}_2\|_{{\mathcal{B}_{\tau}}}\leq C\left(\|\bar{\bl{v}}_1\|_{\mathcal{B}_\tau}+\|\bar{\bl{v}}_2\|_{{\mathcal{B}_{\tau}}}\right)\|\bar{\bl{v}}_1-\bar{\bl{v}}_2\|_{{\mathcal{B}_{\tau}}}\leq 4C^2\e\|\bar{\bl{v}}_1-\bar{\bl{v}}_2\|_{{\mathcal{B}_{\tau}}}\,.
\end{split}
\ee
Here the constant $C$ may be different with that in \eqref{nonlinear9}, but we can still unify them by choosing the larger one in both inequalities. Thus \eqref{nlin10} shows the operator $\mathrm{T}$ is a contract mapping since $\e<(4C^2)^{-1}$. The \emph{Banach fixed point theorem} indicates the equation
\[
\mathrm{T}\,\bl{v}=\bl{v}
\]
has \emph{a unique solution} in $\mathbf{B}_{2C\e}$. And it satisfies
\[
\|\bl{v}\|_{{\mathcal{B}_{\tau}}}\leq C\left(\|\bl{f}\|_{\mathcal{E}_{\la_\th,\la_z,\la}}+\|\bl{g}\|_{\mathcal{V}}\right)\,.
\]
This completes the existence proof of the main result.

\subsection{Non-uniqueness of the case for $\nu<-2$ }
from the existence proof of Theorem \ref{Main} in the case of $\nu<-2$, we can choose $\t{\mu}$, which is different but sufficiently close to $\mu$, such that
\[
\t{\bl{u}}=\f{\nu}{ r}\bl{e}_r+\f{\t{\mu}}{r}\bl{e}_\th+\t{\bl{v}},\q \t{\bl{v}}\in \mathcal{B}_{\tau}\,,
\]

solves the reduced problem
\[
\left\{\begin{split}
&-\left(\p_r^2+\f{1-\nu}{r}\p_r-\f{1-\nu}{r^2}+\p_z^2\right)\t{v}_r+\p_r\t{\pi}=-\left(\t{v}_r\p_r+{\t{v}_z}\p_z\right)\t{v}_r+\f{\t{v}_\th^2}{r}+{\f{2\t{\mu}}{r^2}\t{v}_\th}+f_r\,,\\
& -\left(\p_r^2+\f{1-\nu}{r}\p_r-\f{1+\nu}{r^2}+\p_z^2\right)\t{v}_\th=-\left(\t{v}_r\p_r+{\t{v}_z}\p_z\right)\t{v}_\th-\f{\t{v}_r\t{v}_\th}{r}+f_\th\,,\\
&-\left(\p_r^2+\f{1-\nu}{r}\p_r+\p_z^2\right)\t{v}_z+\p_z\t{\pi}=-\left(\t{v}_r\p_r+{\t{v}_z}\p_z\right)\t{v}_z+f_z\,,\\
&\p_z(r \t{v}_z)+\p_r(r\t{v}_r)=0\,,\\
&\t{v}_r\big|_{r=1}={g}_r,\q \t{v}_\th\big|_{r=1}= g_\th-\t{\mu}+\mu,\q \t{v}_z\big|_{r=1}= g_z, \q \t{\bl{v}}\big|_{r\rightarrow+\i}=0\,.
\end{split}\right.
\]
 is still a solution of system \eqref{ns0}. Since $\bl{v},\,\t{\bl{v}}\in\mathcal{B}_\tau$, we have
  \[
  v_\th\,,\,\t{v}_\th\in o(r^{-1})\,,\q\text{as}\q r\to\i\,.
  \]
Thus we find $\bl{u}$ and $\t{\bl{u}}$ are indeed two distinct solutions of \eqref{ns0} since their $r^{-1}$ coefficients are different. $\hfill\square$

\section*{Data availability statement}
Data sharing is not applicable to this article as no datasets were generated or analysed during the current study.

\section*{Conflict of interest statement}
The authors declare that they have no conflict of interest.
\section*{Acknowledgments}
\q Z. Li is supported by the China Postdoctoral Science Foundation (No. 2024M763474) and the National Natural Science Foundation of China (No. 12001285). X. Pan is supported by the National Natural Science Foundation of China (No. 12031006, No. 12471222).


\begin{thebibliography}{1}

\bibitem{AS1992}
{\sc M. Abramowitz and I. A. Stegun}, Handbook of Mathematical Functions with Formulas, Graphs, and Mathematical Tables, Reprint of the 1972 edition, Dover Publications, Inc.,
New York, 1992.

\bibitem{CarrilloPZ:2020JFA} {\sc B. Carrillo, X. Pan and Q. S. Zhang}, Decay and vanishing of some axially symmetric D-solutions of the Navier-Stokes equations, J. Funct. Anal. 279 (2020), no. 1, 108504, 49 pp.

\bibitem{CarrilloPZZ:2020ARMA} {\sc B. Carrillo, X. Pan, Q. S. Zhang and N. Zhao}, Decay and vanishing of some D-solutions of the Navier-Stokes equations, Arch. Ration. Mech. Anal. 237 (2020), no. 3, 1383--1419.

\bibitem{ChenSYT:2008IMRN} {\sc C. C. Chen, R. M. Strain, H. T. Yau and  T. P. Tsai}, Lower bound on the blow-up rate of the axisymmetric Navier-Stokes equations, Int. Math. Res. Not. IMRN 9 (2008).

\bibitem{ChenSYT:2009CPDE} {\sc C. C. Chen, R. M. Strain, T. P. Tsai and H. T. Yau}, Lower bounds on the blow-up rate of the axisymmetric Navier-Stokes equations, II, Comm. Partial Differential Equations 34 (1-3) (2009), 203--232.
\bibitem{ChenFZ:2017DCDS} {\sc H. Chen, D. Fang and T. Zhang}, Regularity of 3D axisymmetric Navier-Stokes equations, Discrete Contin. Dyn. Syst. 37 (2017), no. 4, 1923--1939.

\bibitem{FinnS:1967ARMA} {\sc R. Finn and D. R. Smith}, On the stationary solutions of the Navier-Stokes equations in two dimensions, Arch. Rational Mech. Anal. 25 (1967), 26--39.


\bibitem{Galdi:2011} {\sc G. P. Galdi}, An introduction to the mathematical theory of the Navier-Stokes equations. Steady-state problems, Second edition, 2011.

\bibitem{GallagherHM:2019MATHNA} {\sc I. Gallagher, M. Higaki and Y. Maekawa}, On stationary two-dimensional flows around a fast rotating disk, Math. Nachr. 292 (2019), no. 2, 273-308.



\bibitem{GH2024} {\sc Z. Guo and M. Hillairet}, Extension of Hamel paradox for the 2D exterior Navier-Stokes problem, arXiv:2412.17399v1.


\bibitem{Higaki2023} {\sc M.~Higaki}, Existence of planar non-symmetric stationary flows with large flux in an exterior disk, Journal of Differential Equations, 360  (2023), pp.~182--200.

\bibitem{HigakiMN:2018ARMA} {\sc M. Higaki, Y. Maekawa and Y. Nakahara}, On stationary Navier-Stokes flows around a rotating obstacle in two-dimensions, Arch. Ration. Mech. Anal. 228 (2018), no. 2, 603-651.

\bibitem{Higaki:ARXIV24} {\sc M. Higaki}, Axisymmetric steady Navier-Stokes flows under suction, arXiv: 2404.02854.

\bibitem{Hillairet2013} {\sc M.~Hillairet and P.~Wittwer}, On the existence of solutions to the  planar exterior {N}avier-{S}tokes system, Journal of Differential Equations, 255 (2013), pp.~2996--3019.


\bibitem{KochNSS:2009ACTAMATH} {\sc G. Koch, N. Nadirashvili, G. A. Seregin and V. Sver\'{a}k}, Liouville theorems for the Navier-Stokes equations and applications, Acta Math. 203 (2009), no. 1, 83--105.

\bibitem{KorobkovPR:2019ARMA} {\sc M. V. Korobkov, K. Pileckas and R. Russo}, On convergence of arbitrary D-solution of steady Navier-Stokes system in 2D exterior domains, Arch. Ration. Mech. Anal. 233 (2019), no. 1, 385--407.


\bibitem{KorobkovPR:2020JDE} {\sc M. V. Korobkov, K. Pileckas and R. Russo}, On the steady Navier-Stokes equations in 2D exterior domains, J. Differential Equations 269 (2020), no. 3, 1796--1828.

\bibitem{KorobkovPR:2021ADV}  {\sc M. V. Korobkov, K. Pileckas and R. Russo}, Leray's plane steady state solutions are nontrivial, Adv. Math. 376 (2021), Paper No. 107451, 20 pp.


\bibitem{KorobkovR:2021ARMA} {\sc M. V. Korobkov and X. Ren}, Uniqueness of plane stationary Navier-Stokes flow past an obstacle, Arch. Ration. Mech. Anal. 240 (2021), no. 3, 1487--1519.


\bibitem{KorobkovR:2022JMPA} {\sc M. V. Korobkov and X. Ren}, Leray's plane stationary solutions at small Reynolds numbers, J. Math. Pures Appl. (9) 158 (2022), 71--89.

\bibitem{KozonoTW:2022JFA} {\sc H. Kozono, Y. Terasawa and Y. Wakasugi}, Asymptotic properties of steady solutions to the 3D axisymmetric Navier-Stokes equations with no swirl, J. Funct. Anal. 282 (2022), no. 2, Paper No. 109289, 21 pp.

\bibitem{KozonoTW:2023JDE} {\sc H. Kozono, Y. Terasawa and Y. Wakasugi}, Asymptotic behavior and Liouville-type theorems for axisymmetric stationary Navier-Stokes equations outside of an infinite cylinder with a periodic boundary condition, J. Differential Equations 365 (2023), 905--926.


\bibitem{Lady:1968} {\sc O. A. Ladyzenskaja}, Unique global solvability of the three-dimensional Cauchy problem for the Navier-Stokes equations in the presence of axial symmetry, Zap. Naucn. Sem. Leningrad. Otdel. Mat. Inst. Steklov. (LOMI) 7 (1968) 155--177 (in Russian).

\bibitem{LeiZ:2011JFA} {\sc Z. Lei and Q. S. Zhang}, A Liouville theorem for the axially-symmetric Navier-Stokes equations, J. Funct. Anal. 261 (2011), no. 8, 2323--2345.

\bibitem{LeiZ:2017PJM} {\sc Z. Lei and Q. S. Zhang}, Criticality of the axially symmetric Navier-Stokes equations, Pacific J. Math. 289 (2017), no. 1, 169--187.


\bibitem{Leray:1933JMPA} {\sc J. Leray}, \'{E}tude de diverses \'{e}quations int\'{e}grales non lin\'{e}aires et de quelques probl\`{e}mes que pose l'hydrodynamique, (French)  J. Math. Pures Appl., 12 (1933), 82 pp.

\bibitem{LP2024} {\sc Z. Li and X. Pan}, Existence of the planar stationary flow in the presence of interior sources and sinks in an exterior domain, arXiv:2412.18474 (2024), 19pp.

\bibitem{LiP:2023CCM} {\sc Z. Li and X. Pan}, Asymptotic properties of generalized D-solutions to the stationary axially symmetric Navier-Stokes equations, Commun. Contemp. Math. 25 (2023), no. 5, Paper No. 2250013, 26 pp.


\bibitem{LiPY:2025CVPDE} {\sc Z. Li, X. Pan and J. Yang}, Characterization of smooth solutions to the Navier-Stokes equations in a pipe with two types of slip boundary conditions,  Calc. Var. Partial Differential Equations 64 (2025), no. 3, Paper No. 105.


\bibitem{LiPYZZZ:2024JFA} {\sc Z. Li, X. Pan, X. Yang, C. Zeng, Q. S. Zhang and N. Zhao}, Finite speed axially symmetric Navier-Stokes flows passing a cone, J. Funct. Anal. 286 (2024), no. 10, Paper No. 110393, 116 pp.


\bibitem{Pan:2016JDE} {\sc X. Pan}, Regularity of solutions to axisymmetric Navier-Stokes equations with a slightly supercritical condition, J. Differential Equations 260 (2016), no. 12, 8485--8529.

 \bibitem{PanL:2020NARWA} {\sc X. Pan and Z. Li},  Liouville theorem of axially symmetric Navier-Stokes equations with growing velocity at infinity, Nonlinear Anal. Real World Appl. 56 (2020), 103159, 8 pp.

\bibitem{SereginW:2019AIA} {\sc G. Seregin and W. Wang}, Sufficient conditions on Liouville type theorems for the 3D steady Navier-Stokes equations, Algebra i Analiz 31 (2019), no. 2, 269--278; reprinted in St. Petersburg Math. J. 31 (2020), no. 2, 387--393

\bibitem{UI:1968} {\sc M. R. Ukhovskii and V. I. Iudovich}, Axially symmetric ?ows of ideal and viscous fuids ?lling the whole space, Prikl. Mat. Mekh. 32 (1968) 59--69 (in Russian), translated as J. Appl. Math. Mech. 32 (1968) 52--61.

\bibitem{Wang:2019JDE} {\sc W. Wang}, Remarks on Liouville type theorems for the 3D steady axially symmetric Navier-Stokes equations, J. Differential Equations 266 (2019), no. 10, 6507--6524.

\bibitem{Wei:2016JMAA} {\sc D. Wei}, Regularity criterion to the axially symmetric Navier-Stokes equations, J. Math. Anal. Appl. 435 (2016), no. 1, 402--413.

\bibitem {Yudovich:2003MMJ}  {\sc V. I. Yudovich}, Eleven great problems of mathematical hydrodynamics, Dedicated to Vladimir I. Arnold on the occasion of his 65th birthday. Mosc. Math. J. 3 (2003), no. 2, 711--737, 746.

\bibitem {Yudovich:1967MATHSB}   {\sc V. I. Yudovich}, An example of the loss of stability and the generation of a secondary flow of a fluid in a closed container, Mat. Sb. (N.S.) 74(116) (1967), 565--579.


\end{thebibliography}
\end{document}